\documentclass[10pt]{article}

\usepackage[ngerman,french,english]{babel}

\usepackage[left]{showlabels}                                           % show labels

\usepackage{pdfsync}

\usepackage{amsmath}                                                        % need for subequations
\numberwithin{equation}{section}
\usepackage{graphicx}                                                       % need for figures
\usepackage{subfigure}
\usepackage{enumitem}                                                    % need for subfigures
\usepackage{hyperref}
\usepackage{amssymb}                                                        % gives you \mathbb{} font
\usepackage[mathscr]{eucal}                                             % gives you \mathscr font
\usepackage{cancel}                                                             % gives you the ability to visibly cross out terms in equations}
\usepackage[normalem]{ulem}                                                                 % gives you \sout
\usepackage{pstricks}
\usepackage{rotating}
\usepackage{lscape}
\usepackage[paperwidth=8.5in,paperheight=11in,top=1.00in, bottom=1.00in, left=1.00in, right=1.00in]{geometry}
\usepackage{mathtools}                                                      % need for `show only references'
\mathtoolsset{showonlyrefs=true}                                    % only equations which are labeled AND referenced will be numbered.
\usepackage{fixltx2e,amsmath}                                           % Supposedly, this allows one to use \eqref{} in \caption{}.
\MakeRobust{\eqref}

\usepackage{dsfont}

\usepackage{mathdots}
\usepackage{amsthm}                                                             % need for theorem-proof environment
\allowdisplaybreaks

\theoremstyle{plain}
\newtheorem{theorem}{Theorem}
\numberwithin{theorem}{section}

\newtheorem{lemma}[theorem]{Lemma}                              % [theorem] ==> theorems and lemmas will share a counter
\newtheorem{proposition}[theorem]{Proposition}

\theoremstyle{definition}
\newtheorem{definition}[theorem]{Definition}

\newtheorem{notation}[theorem]{Notation}
\newtheorem{remark}[theorem]{Remark}
\newtheorem{assumption}[theorem]{Assumption}

\usepackage{color}
\DeclareMathOperator*{\esssup}{ess\,sup}

\def \a {{\alpha}}
\def \b {{\beta}}
\def \d {{\delta}}

\def \G {{\Gamma}}

\def \R {{\mathbb {R}}}
\def \N {{\mathbb {N}}}

\def \x {{\xi}}
\def \e {{\varepsilon}}
\def \eps {{\varepsilon}}

\def \t {{\tau}}
\def \t {{\tau}}
\def \n {{\nu}}
\def \m {{\mu}}
\def \y {{\eta}}

\def \tt {{\bar t}}

\def \phi {{\varphi}}

\def \tilde {\widetilde}

\def\p{\partial}

\def \a {{\alpha}}
\def \b {{\beta}}
\def \d {{\delta}}

\def \G {{\Gamma}}

\def \rn  {{\mathbb {R}}^{N}}
\def \R  {{\mathbb {R}}}
\def \x {{\xi}}

\def \e {{\varepsilon}}
\def \eps {{\varepsilon}}
\def \t {{\tau}}
\def \n {{\nu}}
\def \m {{\mu}}
\def \y {{\eta}}

\def \p {{\partial}}
\def \a {{\alpha}}

\def \d {{\delta}}

\def \a {{\alpha}}
\def \b {{\beta}}
\def \d {{\delta}}

\def \G {\Ga}

\def \Ga {{\Gamma}}

\def \ddd {{\bar{d}}}

\def \SS {\mathcal{S}}

\def \R {{\mathbb {R}}}
\def \N {{\mathbb {N}}}

\def \x {{\xi}}
\def \e {{\varepsilon}}
\def \eps {{\varepsilon}}

\def \t {{\tau}}
\def \t {{\tau}}
\def \n {{\nu}}

\def \m {{\mu}}
\def \y {{\eta}}

\def \tt {{\bar t}}

\def \phi {{\varphi}}

\def \tilde {\widetilde}

\def \A {\mathcal{A}}

\def \rn {{\mathbb {R}}^{N}}

\def \à {{\`a }}
\def \è {{\`e }}
\def \ò {{\`o }}
\def \ù {{\`u }}

\newcommand\Ac{\mathscr{A}}

\newcommand\Lc{\A+Y}

\newcommand{\<}{\langle}
\renewcommand{\>}{\rangle}
\renewcommand{\(}{\left(}
\renewcommand{\)}{\right)}

\def \sol {p} %fundamental solution
\def \gg {\mathbf{G}} %gaussian function
\def \param {\mathbf{P}}
\def \C {\mathcal{C}}
\def \TT {T_0} %estremo dominio temporale
\def \rr {q} %set di variabili %%{\rr }
\def \T {T}
\def \xx {{\x}}

\makeatletter
\def\section{\@startsection {section}{1}{\z@}{3.25ex plus 1ex minus
 .2ex}{1.5ex plus .2ex}{\large\bf}}
\def\subsection{\@startsection{subsection}{2}{\z@}{3.25ex plus 1ex minus
 .2ex}{1.5ex plus .2ex}{\normalsize\bf}}
\@addtoreset{equation}{section} %%
\makeatother
\begin{document}

\title{%Existence and
Optimal regularity for degenerate Kolmogorov equations with rough %under a parabolic H\"ormander
coefficients
%Parametrix technique and martingale problem
%for stochastic equations with measurable coefficients
%under the parabolic H\"ormander condition
}

\author{
Giacomo Lucertini
\thanks{Dipartimento di Matematica, Universit\`a di Bologna, Bologna, Italy.
\textbf{e-mail}: giacomo.lucertini3@unibo.it.}
\and Stefano Pagliarani
\thanks{Dipartimento di Matematica, Universit\`a di Bologna, Bologna, Italy.
\textbf{e-mail}: stefano.pagliarani9@unibo.it} \and Andrea Pascucci
\thanks{Dipartimento di Matematica, Universit\`a di Bologna, Bologna, Italy.
\textbf{e-mail}: andrea.pascucci@unibo.it}
}

\date{This version: \today}

\maketitle

\begin{abstract}
{We consider a class of degenerate equations satisfying a parabolic H\"ormander condition, with
coefficients that are measurable in time and H\"older continuous in the space variables. By
utilizing a generalized notion of strong solution, we establish the existence of a fundamental
solution and its optimal H\"older regularity, as well as Gaussian estimates. These results are key
to study the backward Kolmogorov equations associated to a class of Langevin-type diffusions.
%In this setting it is necessary to introduce a generalized notion of strong solution:
%the equation is solved almost everywhere along the integral curve of the transport term of the operator.
%By adapting the classical Levi's parametrix method,
%we prove the existence of a fundamental solution and establish optimal regularity estimates, as well as Gaussian upper/lower bounds.
%and maximal regularity of such solutions as well as Gaussian estimates.
%These operators can be viewed as the backward Kolmogorov operators associated to a family of SDEs.
}
%\\ \\
%\red{We consider a class of degenerate equations satisfying a parabolic H\"ormander condition,
%with coefficients that are measurable in time and H\"older continuous in the space variables. We
%prove the existence, estimates, and qualitative properties for the fundamental solutions of these
%operators. In order to do that, we adapt the classical Levi's parametrix method. These operators
%can be viewed as the backward Kolmogorov operators associated to a family of SDEs.
%Using the previous results and solving the martingale problem, we prove the existence and the uniqueness in
%the weak sense of the solution for such equations.
%}
\end{abstract}


\vspace{2pt}


%%%%%%%%%%%%%%%%%%%%%%%%%%%%%%%%%%%%%%%
%
%       SECTION: Introduction
%
%%%%%%%%%%%%%%%%%%%%%%%%%%%%%%%%%%%%%%%
\section{Introduction}
%We %prove existence and
%The main purpose of this paper is to study the optimal regularity properties of the fundamental
%solution to a Kolmogorov operator under a parabolic H\"ormander-type condition
%and assuming that the coefficients are %regular %in the space variables but only
%measurable in the time variable. More precisely, we consider the second order differential
%operator in non-divergence form
%The main purpose of this paper is to
We study existence and optimal regularity properties of the fundamental solution to a Kolmogorov
operator that satisfies a parabolic H\"ormander-type condition. The coefficients of the operator
are H\"older continuous in the space variables but only measurable in time.

{Precisely, for fixed $d\le N$ and $\TT>0$, we consider the second order operator %on $\mathcal{S}_{\TT}:=[0,\TT]\times \R^N$
in non-divergence form
%\begin{equation}
 $\Ac + Y$
%\end{equation}
with
\begin{equation}
\begin{split}\label{Lter}
 \Ac & = \frac{1}{2}\sum_{i,j=1}^{d} a_{ij}(t,x)\p_{x_i x_j}+ \sum_{i=1}^{d} a_{i}(t,x)\p_{x_i}
 +a(t,x),\qquad (t,x)\in \mathcal{S}_{\TT}:=]0,\TT[\times \R^N, \\
  Y  &   = \p_t + \langle B x,\nabla\rangle = \p_{t}  +  \sum_{i,j=1}^{N}b_{ij}x_{j}\p_{x_{i}}  , \qquad x\in \R^N,
\end{split}
\end{equation}
where $B$ is a constant matrix of dimension $N\times N$. Here, $\Ac$ is an elliptic operator on
$\R^d$ and $Y$ is a first order differential operator on $\R\times \R^N$, also called {\it
transport} or {\it drift term}. The focus of this paper is mainly on the case $d<N$, that is when
$\Ac + Y$ is fully degenerate, namely no coercivity condition on $\R^N$ is satisfied.} %More
%explicitly said: $\Ac + Y$ contains second order derivatives only w.r.t. to the the first $d$
%space variables.
%We are interested in the backward Kolmogorov operator
%%\begin{equation}\label{L}
% $\Ac + Y$
%%\end{equation}
%where $Y$ is the first order differential operator, also called {\it transport or drift term},
%\begin{equation}\label{L_bis}
% Y=\langle B x,\nabla\rangle +\p_t=\sum_{i,j=1}^{N}b_{ij}x_{j}\p_{x_{i}}+\p_{t}
%\end{equation}
%and $B$ is a constant matrix of dimension $N\times N$.

Motivations for the study of %operators of the form
$\Ac + Y$ come from physics and  finance. In its most basic form, with $N=2$  and $d=1$,
\begin{equation}\label{ae2}
  \frac{1}{2}\p_{x_{1}x_{1}}+x_{1}\p_{x_{2}}+\p_{t}
\end{equation}
is the backward Kolmogorov operator of the system of stochastic equations
\begin{equation}\label{ae3}
  \begin{cases}
    dV_{t}=dW_{t}&\\
    dX_{t}=V_{t}dt&
  \end{cases}
\end{equation}
where $W$ is a real Brownian motion. In the classical Langevin model, $(V,X)$ describes the
velocity and position of a particle in the phase space and is the prototype of more general
kinetic models (cf. \cite{MR1972000}, \cite{MR4275241}, \cite{MR4229202}). In mathematical
finance, $(V,X)$ represents the log-price and average processes used in modeling path-dependent
financial derivatives, such as Asian options (cf. \cite{MR1830951}, \cite{MR3660883}). The study
of the fundamental solution and its regularity properties is a crucial step in tackling the
martingale problem for the corresponding stochastic equations, particularly for well-posedness and
pathwise uniqueness problems. These issues will be addressed in a future work.

\subsection{Main assumptions}\label{sec:assumptions}
Throughout the paper, $\Ac + Y$ verifies the following two structural
\begin{assumption}[\bf Coercivity on $\R^{d}$]\label{coer}
%\end{assumption}
%For a given $\TT>0$ we have what follows. The second-order coefficient matrix of $\Ac$, namely $A_0(t,\cdot):=\(a_{ij}(t,\cdot)\)_{i,j=1,\dots,d_0}$, is symmetric and %positive definite
%The coefficients $a_{ij}=a_{ji},a_i,c$, for $1\le i,j\le d$, are measurable functions of
%$(t,x)\in[0,\TT]\times\mathbb{R}^d$ and
For $1\le i,j\le d$, the coefficients $a_{ij},a_{i},a$ are in
$L^\infty([0,\TT];C_{b}(\R^N))$, where $C_{b}(\R^N)$ denotes the space of bounded continuous functions on $\rn$. %$L_{\text{\rm loc}}^\infty(\mathcal{S}_T)$.
The diffusion matrix $\(a_{ij}\)_{i,j=1,\dots,d}$ is symmetric and
there exists a positive constant $\m$ such that %$\m\geq1
\begin{equation}
 {\m^{-1}}{|\x|^2} \leq \sum_{i,j=1}^{d}a_{ij}(t,x)\x_i \x_j \leq \m |\x|^2,\qquad x\in\R^N,\
 \x\in\R^{d},
\end{equation}
for almost every $t\in[0,\TT]$.
%\\
%\red{Alternativa:
%\begin{equation}
%\m^{-1}|\x|^2 \leq \left\|\sum_{i,j=1}^{d_0}a_{ij}(\cdot,x)\x_i \x_j\right\|_{L^\infty([0,\TT])} \leq \m |\x|^2,
%\qquad %\forall
%x\in\R^N,\x\in\R^{d_0},
%\end{equation}}
\end{assumption}
\begin{assumption}[\bf H\"ormander condition]\label{ass:hypo}
The vector fields $\partial_{x_1},\dots,\partial_{x_{d}}$ and $Y$ satisfy
\begin{equation}\label{horcon}
\text{rank } \text{Lie}(\partial_{x_1},\dots,\partial_{x_{d}},Y)  = N+1.
\end{equation}
\end{assumption}
%Further assumptions concerning the regularity of the coefficients with respect to the spatial
%variables will be specified later (cf. Assumption \ref{ass:regularity}).
We refer to \eqref{horcon} as a {\it parabolic} H\"ormander condition since the drift term $Y$
plays a key role in the generation of the Lie algebra. Under Assumption \ref{ass:hypo}, the
{prototype} %ical%al
Kolmogorov operator %with $\A=\frac{1}{2}\Delta_{d}$ that is
\begin{equation}\label{eq:kolm_const}
 %\Kc=
 \frac{\d}{2}\sum_{i=1}^{d}\p_{x_{i}x_{i}} + Y %, \qquad \lambda>0,%\underbrace{\langle B x, \nabla_x \rangle + \partial_t}_{=:Y} ,\qquad   x\in \mathbb{R}^N,
\end{equation}
is hypoelliptic for any $\d>0$. Kolmogorov \cite{MR1503147} (see also \cite{MR222474}) constructed
the explicit Gaussian fundamental solution for \eqref{eq:kolm_const}, which is the transition
density of a linear stochastic differential equation.
%(see Sections \ref{sec:stoc_processes} and \ref{sec:kolmogorov_operators} for more details).

Condition \eqref{horcon} is equivalent to the well-known Kalman rank condition for controllability
in linear systems theory (cf., for instance, \cite{MR2791231}). Also, it was shown in
\cite{lanpol} that \eqref{horcon} is equivalent to $B$ having the block-form
\begin{equation}\label{B}
  B=\begin{pmatrix}
 \ast & \ast & \cdots & \ast & \ast \\
 B_1 & \ast &\cdots& \ast & \ast \\
 0 & B_2 &\cdots& \ast& \ast \\
 \vdots & \vdots &\ddots& \vdots&\vdots \\
 0 & 0 &\cdots& B_{\rr}& \ast
  \end{pmatrix}
\end{equation}
where the $*$-blocks are arbitrary and $B_j$ is a $(d_{j-1}\times d_j)$-matrix of rank $d_j$ with
\begin{equation}
d\equiv d_{0}\geq d_1\geq\dots\geq d_{\rr}\geq1,\qquad \sum_{i=0}^{\rr} d_i=N.
\end{equation}
%\end{remark}
%Remark \ref{blocks} enables us to regard $\A + Y$ as a Langevin-type operator with variable
%coefficients, as it is suggested by the following
%\begin{example}\label{ex:langevin}
%Indeed, in the case of $N=2$, $d_0 = 1$ and
%\begin{equation}\label{B_langevin}
%  B=\begin{pmatrix}
% 0 & 0  \\
%1 & 0
%  \end{pmatrix},
%\end{equation}
%the operator $\A + Y$ reads as
%\begin{equation}
%\A + Y = \frac{\s^{2}(t,v,\xi)}{2}  \partial_{vv} + v \partial_{\xi} + \partial_t, \qquad
%(t,v,\xi)\in \R\times\R^2,
%\end{equation}
%which is the Kolmogorov operator of the stochastic kinetic model
%\begin{equation}\label{eq:langevin_SDE_var}
%\begin{cases}
%\dd V_t & = \sigma(t,V_t, \Xi_t) \dd W_t \\
%\dd \Xi_t & = V_t \dd t
%\end{cases}.
%\end{equation}
%Here, $V$ represents the velocity of a particle subject to a stochastic force and $X$ its position
%in space. For constant $\s$, $\A + Y$ reduces to a Kolmogorov operator of type
%\eqref{eq:kolm_const} and its fundamental solution was first explicitly characterized in \cite{}.
%\end{example}
%The block-structure \eqref{B} induces
This allows to introduce natural definitions of anisotropic norm and H\"older continuity  on
$\R^{N}$.
%that is natural for the study of $\mathcal{A}_{t}$:
%\begin{notation}
%For any $T>0$ we set $\mathcal{S}_T:=[0,T[\times\R^N$.
%\end{notation}
\begin{definition}[\bf Anisotropic norm and H\"older spaces]\label{def:holder_coeff}
%Let $\TT>0$.
For any $x\in\rn$ let
\begin{equation}\label{e7}
 %\|{(t,x)}\|_B:=|t|^{1/2}+|x|_B,\qquad
 |x|_B:=%\sum_{j=1}^d |x_j|^{1/q_j}=
 \sum_{j=0}^{\rr}\, \sum_{i=\bar{d}_{j-1}+1}^{\bar{d}_j} |x_{i}|^{\frac{1}{2j+1}}, \qquad %\text{where}\quad %\bar{p}_j:=  p_0+\dots + p_j
{\bar{d}_j:= \sum_{k=0}^j d_k}.%,\qquad  {\bar{d}_{-1}:=0}.%{\blue , \ {\bar{d}_{0}:=d}}.%.{\blue \qquad (t,x)\in\R\times\R^d,}
\end{equation}
For $\alpha\in\,]0,1]$ we denote by $C^{\alpha}_{B}(\rn)$ the set of functions $g\in C_{b}(\rn)$ such that %set of $f:\R^d \to \R$ such that
\begin{equation}
 %\| \phi \|_{C^{\alpha}_{B}(\rn)} =
 \| g\|_{C^{\alpha}_{B}(\rn)} :=\sup_{x\in\rn} |g(x)|+
\sup_{x,y\in\rn\atop x\neq y} \frac{|g(x) - g(y)| }{ |x - y|^{\alpha}_B } < \infty,
\end{equation}
and by $L^{\infty}([0,\TT];C^{\alpha}_{B}(\rn))$ %{\blue(mettiamo $[0,\TT[$ per consistenza?)}
the set of %Borel
measurable functions $f:[0,\TT]
\longrightarrow C^{\alpha}_{B}(\rn)$ %, considering the standard Borel $\s$-algebra on $C_b(\R^d)$,
such that
\begin{equation}
 %\| f \|_{L^{\infty}([0,\TT] ;C^{\alpha}_{B}(\rn))} =
 \| f \|_{L^{\infty}( [0,\TT] ;C^{\alpha}_{B}(\rn))} := \esssup_{t\in[0,\TT]} \| f(t) \|_{C^{\alpha}_{B}(\rn)} < \infty.
\end{equation}
\end{definition}
The quasi-norm \eqref{e7}, hereafter referred to as to \emph{intrinsic norm} for $\A+Y$, % with a slight abuse of terminology,
can be directly related to the scaling properties of the underlying
diffusion process (cf. \cite{MR2659772}, \cite{lanconelli2020local}). For example, %we have $N=2$, $d_{0}=d_{1}=1$, $B_{1}=1$, the
%$\ast$-blocks are equal to $0$ and %. In the particular case of the Langevin system in Example \ref{ex:langevin},
the intrinsic norm for the Langevin operator \eqref{ae2} reads as
%\begin{equation}
$|(v,x)|_B = |v| + |x|^{\frac{1}{3}}$ for $(v,x)\in\R^{2}$
%\end{equation}
and reflects the time-scaling properties of the stochastic system \eqref{ae3}, i.e. $(\Delta V )^2
\approx \Delta t$ and $(\Delta X)^2 \approx (\Delta t)^3$. %Notice that $C^{0}_{B}$ is simply the
%set of bounded and continuous functions on $\rn$.
The regularity of the coefficients of $\A$ is stated in our last standing
\begin{assumption}\label{ass:regularity}
The coefficients $a_{ij},a_i,a$ of $\A$ belong to $L^{\infty}([0,\TT];C^{\alpha}_{B}(\rn))$ for
some $\alpha\in\, ]0,1]$.
\end{assumption}
Throughout paper we suppose that Assumptions \ref{coer}, \ref{ass:hypo} and \ref{ass:regularity}
are fulfilled.

\subsection{Main results}\label{sec:main_results}
According to Assumption \ref{ass:regularity}, the coefficients of $\A$ are intrinsically H\"older
continuous in the space variables and merely measurable in the time component. For Kolmogorov
operators with coefficients that are H\"older continuous in both space and time, the study of the
existence of a fundamental solution goes back to the early papers \cite{weber}, \cite{ilin},
\cite{Sonin} and \cite{MR217739}. A modern and more natural approach based on the Lie group theory
was developed by \cite{pol}, \cite{difpas}, \cite{francesca2021fundamental} and \cite{pagliarani2021yosida}. Applications to the
martingale problem for some degenerate diffusion processes are given in \cite{menozzi} and
\cite{MR3758337}.

Major questions in the study of Kolmogorov equations are the very definition of solution and its
optimal regularity properties. It is well-known that, in general, the fundamental solution is not
regular enough to support the derivatives $\p_{x_{i}}$, for $d<i\le N$, appearing in the transport
term $Y$. Indeed, under the H\"ormander condition \eqref{horcon}, these derivatives are of order
three and higher in the intrinsic sense. For this reason, even for equations with H\"older
coefficients, weak notions of solution have been introduced. %and characterized in terms of their regularity properties. %Roughly speaking,
In this regard we may identify two main streams of research. In the semigroup approach initiated
by \cite{MR1475774}, solutions are defined in the {\it distributional} sense: in this framework,
solutions do not benefit from the time-smoothing effect that is typical of parabolic equations
(see, for instance, Theorem 4.3 in \cite{MR2534181}). On the other hand, in the stream of research
started by \cite{pol}, solutions in the {\it Lie} sense are defined by regarding $Y$ as a
directional derivative. In this approach regularity properties in space and time are strictly
intertwined: this allows to fully exploit the smoothing effect of the equation but makes the
analysis less suitable for applications to stochastic equations.

Recently, a third notion of solution, which is a cross between the two previous ones, %{\blue (?)},
has been proposed in \cite{MR4355925} with the aim of studying %Kolmogorov {\blue (Langevin ?)}
Langevin {\it stochastic PDEs} with rough coefficients. Since we are specifically interested in
operators whose coefficients are only measurable in time, it seems natural to adopt this latter
approach for our analysis. We introduce the following definition that is a particular case of
(1.3) in \cite{MR4355925} when $N=2$.
%We define the integral solution of the equation $Yu=-\A u$ along the integral curves of the vector field $Y$:
\begin{definition}[\bf Strong Lie solution]
\label{solint}
Let $0<T\leq \TT$ and $f\in L_{\text{\rm loc}}^{1}([0,T[;C_{b}(\rn))$. %$g\in L_{\text{\rm loc}}^\infty(\mathcal{S}_T)$.
A solution to equation
\begin{equation}\label{ae1}
  \A u+ Yu = f\ \text{ on }\mathcal{S}_{T}
\end{equation}
%:=[0,T[\times\R^N$, $T>0$,
%when $u\in L^{\infty}_{loc}\([0,T[; C^{2,\a}_{B}\)$
is a %Borel-measurable
continuous function $u$ such that {there exist} $\p_{x_{i}}u,\p_{x_i x_j}u\in L_{\text{\rm
loc}}^{1}([0,T[;C_{b}(\rn))$, for $i,j=1,\dots,d$, and
\begin{equation}\label{eq:integral_sol}
 u(s,e^{(s-t)B}x)= u(t,x)-\int_t^{s}\left(\A u(\t,e^{(\t-t)B}x)-f(\t,e^{(\t-t)B}x)\right)d\t,
 \qquad (t, x)\in \mathcal{S}_{T},\ s<T.
%u(\g^{(t,x)}(s))=
%u(t,x)-\int_t^{s}\A u(\g^{(t,x)}(\t))d\t,
%Let $T\in\,]0,\TT]$ and $g\in L_{\text{\rm loc}}^{1}([0,T[;C(\rn))$. %$g\in L_{\text{\rm loc}}^\infty(\mathcal{S}_T)$.
%A solution to equation
%\begin{equation}\label{ae1}
%  \A u+ Yu = g\ \text{ on }\mathcal{S}_T
%\end{equation}
%%:=[0,T[\times\R^N$, $T>0$,
%%when $u\in L^{\infty}_{loc}\([0,T[; C^{2,\a}_{B}\)$
%is a %Borel-measurable
%continuous function $u$ such that $\p_{x_{i}}u,\p_{x_i x_j}u\in L_{\text{\rm loc}}^{1}([0,T[;C(\rn))$,
%for $i,j=1,\dots,d$, and
%\begin{equation}\label{eq:integral_sol}
% u(s,e^{(s-t)B}x)= u(t,x)-\int_t^{s}\left((\A u)(\t,e^{(\t-t)B}x)-g(\t,e^{(\t-t)B}x)\right)d\t,
% \qquad (t, x)\in \mathcal{S}_T,\ s\in[t,T[.
%%u(\g^{(t,x)}(s))=
%%u(t,x)-\int_t^{s}\A u(\g^{(t,x)}(\t))d\t,
\end{equation}
\end{definition}
\begin{remark}
Notice that $s\mapsto(s,e^{(s-t)B}x)$ is the integral curve of $Y$ starting from $(t,x)$: for any
suitably regular function $u$ the limit
\begin{equation}\label{eq:limit_Lie_deriv}
 Yu(t,x):= \lim_{s\to t}\frac{u(s,e^{(s-t)B}x)-u(t,x)}{s-t}
\end{equation}
is the directional (or Lie) derivative along $Y$ of $u$ at $(t,x)$. %Thus, if %we knew a priori that
%$u,g$ and the coefficients are smooth enough, then \eqref{eq:integral_sol} is equivalent to the
%fact that $u$ is a classical (pointwise) solution of \eqref{ae1}.
%%Instead, if $u$ and $g$ has no smoother regularity and $u$ is a strong solution of \eqref{ae1} on $\mathcal{S}_T$,
Thus, if the integrand in \eqref{eq:integral_sol} is continuous then $u$ is a classical
(pointwise) solution of \eqref{ae1}. However, as noticed in Remark \ref{rem:almost_surely}, in
general a solution $u$ in the sense of Definition \ref{solint} is only a.e. differentiable along
$Y$ and equation \eqref{ae1} is satisfied for
almost every $(t,x)\in \mathcal{S}_{\T}$.% {\blue (da
%decidere dove mettere questo remark)}.
\end{remark}

In order to state our first main result, we give the following
\begin{definition}[\bf Fundamental solution]\label{fund}
%Let $\A + Y$ be as in \eqref{L} with coefficients $a_{ij},a_{i},a \in L_{\text{\rm loc}}^\infty(\mathcal{S}_{\TT})$.
A fundamental solution of $\A + Y$ is a function $\sol=\sol(t,x;\T,y)$ defined for $t<\T$ and $x,y\in\R^N$ such that, for any fixed $(\T,y)\in\mathcal{S}_{\TT}$, we have:%\,\mathcal{S}_{\TT}
\begin{itemize}
%1
\item[i)] $\sol(\cdot,\cdot;\T,y)$ is a solution of equation
$\A u+Y u=0$ on $\mathcal{S}_{\T}$ in the sense of Definition \ref{solint};
%2
\item[ii)] for any $g\in C_{b}(\rn)$ we have
\begin{equation}\label{condfin}
\lim_{\substack{(t,x)\to(\T,y)\\ t<\T}} \int_{\R^N}\sol(t,x;\T,\y)g(\y)d\y=g(y).
\end{equation}
\end{itemize}
\end{definition}
The following result states the existence of the fundamental solution $p$ of $\A + Y$, as well as
uniform Gaussian bounds for $p$ and its derivatives with respect to the non-degenerate variables
$x_1,\dots,x_{d}$.
\begin{theorem}[\bf Existence and Gaussian bounds]\label{main}
Under Assumptions \ref{coer}, \ref{ass:hypo} and \ref{ass:regularity}, $\A+Y$
%\begin{equation}
%a_{ij},a_i,a\in %C_{B,T}^{0,\alpha}
%L^{\infty}\big( [0,\TT] ;C^{\alpha}_{B}(\rn)\big),\qquad i,j=1,\dots,d,
%\end{equation}
has a fundamental solution $\sol=\sol(t,x;\T,y)$ in the sense of Definition \ref{fund}. For every
$\e>0$ there exists a positive constant $C$, only dependent on $\TT,\m,B,\e,\a$ %$ \| a_{i} \|_{L^{\infty}( [0,\TT] ;C^{\alpha}_{B}(\rn))},\| a\|_{L^{\infty}( [0,\TT] ;C^{\alpha}_{B}(\rn))}$ %\| a_{ij} \|_{L^{\infty}( [0,\TT] ;C^{\alpha}_{B}(\rn))},
and the $\a$-H\"older norms of the coefficients, such that
\begin{align}\label{eq:gaussian_1}
 %C^{-1} \G^{\frac{1}{\mu}-\e}(t,x;\T,y) \leq
 \sol(t,x;\T,y)&\leq C \G^{\m+\e}(t,x;\T,y),\\ \label{eq:gaussian_2}
 \left|\p_{x_i}\sol(t,x;\T,y)\right| &\leq \frac{C}{\sqrt{\T-t}} \G^{\m+\e}(t,x;\T,y),\\  \label{eq:gaussian_3}% \qquad i=1,\dots, d,\\
 \left|\p_{x_i x_j}\sol(t,x;\T,y)\right| &\leq \frac{C}{\T-t} \G^{\m+\e}(t,x;\T,y),% \qquad
%i,j=1,\dots, d,
\end{align}
for any $(T,y)\in\mathcal{S}_{\TT}$, $(t,x)\in\mathcal{S}_{T}$ %$0\leq t<\T\leq\TT$, $x,y\in\R^N$%\red{$(\T,y)\in \mathcal{S}_{\TT}$}, %$(\T,y)\in\,\mathcal{S}_{\TT}$, $(t,x)\in \mathcal{S}_\T$
and $i,j=1,\dots,d$, where $\G^\d$ is the Gaussian fundamental solution
of \eqref{eq:kolm_const}, whose explicitly expression is given in \eqref{gaussian}. Moreover,
there exist two positive constants $\bar{\m},\bar{c}$ such that
\begin{equation}\label{eq:gaussian_4}
  \bar{c}\, \G^{\bar{\m}}(t,x;\T,y)
 \leq \sol(t,x;\T,y),
\end{equation}
for any $(T,y)\in\mathcal{S}_{\TT}$ and $(t,x)\in\mathcal{S}_{T}$. %$0\leq t<\T\leq\TT$ and $x,y\in\R^N$. %$(\T,y)\in \mathcal{S}_{\TT}$ and $(t,x)\in \mathcal{S}_\T$.
%\end{itemize}
\end{theorem}
In Section \ref{sec:Cau} we present several results
for the Cauchy problem that are straightforward consequences of Theorem \ref{main}. The proof of Theorem \ref{main} %in Section \ref{sec:parametrix}
is based on a modification of Levi's parametrix technique, which allows to deal with the lack of
regularity of the coefficients along the drift term $Y$. The main tool is the fundamental solution
of a Kolmogorov operator with time-dependent measurable coefficients, also recently studied in
\cite{brampol}. This approach allows for a careful analysis of the optimal regularity properties
of the fundamental solution
$p$: %. A result in this direction is provided by
Theorem \ref{ta1} below states that $p$ belongs to the intrinsic H\"older space $C_{B}^{2,\a}$ as
given by Definition \ref{def:C2_alpha_space}. As the notation could be misleading, we explicitly
remark that for $u\in C_{B}^{2,\a}$ not even the first order derivatives $\p_{x_{i}}u$, for $i>d$,
necessarily exist. However, in general we cannot expect higher regularity properties for solutions
to \eqref{ae1} and
%in fact the
$C_{B}^{2,\a}$-regularity is indeed optimal.

\begin{theorem}[\bf %Optimal
Regularity {of the fundamental solution}]\label{ta1}
%Under Assumptions \ref{coer}, \ref{ass:hypo} and \ref{ass:regularity}, the fundamental solution
Under the assumptions of Theorem \ref{main}, $\sol(\cdot,\cdot;\T,y)\in
C_{B}^{2,{\beta}}(\mathcal{S}_{\t})$ for every $(T,y)\in \mathcal{S}_{\TT}$, $0<\t<\T$ and
$\b<\a$. Precisely,
there exists a positive constant $C$ only dependent on $\TT,\m,B,\beta,\a$ %$ \| a_{i} \|_{L^{\infty}( [0,\TT] ;C^{\alpha}_{B}(\rn))},\| a\|_{L^{\infty}( [0,\TT] ;C^{\alpha}_{B}(\rn))}$ %\| a_{ij} \|_{L^{\infty}( [0,\TT] ;C^{\alpha}_{B}(\rn))},
and the $\a$-H\"older norms of the coefficients,  such that
%is such that
\begin{equation}\label{eq:estimte_C2alpha_sol_fond}
  \| p(\cdot,\cdot;\T,y) \|_{C_{B}^{2,\beta}(\mathcal{S}_{\t})} \leq \frac{C}{(T-\t)^{\frac{Q+{2+\b}}{2}}},   %\G^{\m+\e}(t,x;\T,y) .
\end{equation}
{where $Q$ is the so-called \it{homogeneous dimension} of $\R^N$ with respect to the quasi-norm
$|\cdot|_B$ defined by
\begin{equation}\label{ae11}
 Q=\sum_{i=0}^{\rr} (2i+1) d_i .
\end{equation}}
%for any $\bar{\a}<\alpha$, $(\T,y)\in\,\mathcal{S}_{\TT}$ and $K$ compact subset of
%$\mathcal{S}_\T$.
\end{theorem}
{%Also for the operators with regular coefficients studied in \cite{pol} and \cite{difpas},
Theorem \ref{ta1} refines the known results about the smoothness of the fundamental solution (cf.
\cite{MR1475774}, \cite{MR1751429}, \cite{difpol}) and exhibits its maximal regularity
properties.} To give the precise definition of the H\"older space $C_{B}^{2,\alpha}$ %in
%Theorem \ref{main}, %We claim that the result $p(\cdot,\cdot;T,y) \in
%C_{B}^{2,\alpha}(\mathcal{S}_T)$ is optimal under the hypothesis $a_{ij},a_i,a\in L^{\infty}\big(
%[0,\TT];C^{\alpha}_{B}(\rn)\big)$ of Assumption \ref{ass:regularity}. The space
%$C_{B}^{2,\alpha}(\mathcal{S}_T)$ is defined by only specifying the regularity of the function and
%its derivatives along the fields $\p_{x_{1}}, \dots, \p_{x_{d}}$ and $Y$. We shall then compare it
%with its counterpart introduced in \cite{pagpaspig}, \cite{PagPig}, which would include the
%fundamental solution under the stronger assumption of intrinsic space-time H\"olderianity for the
%coefficients.
%because of the lower regularity assumption on the Lie-derivative $Yu$, which is no more than measurable along the field $Y$ itself. Critically, this implies that a weaker intrinsic Taylor formula holds (see Theorem \ref{th:main_tay} below) for functions that belong to this version of $C^{2,\a}( \R \times \R^d)$. We refer to Remark \ref{} below for a comparison between the Taylor formula proved in this paper and the one proved in \cite{PPP1}, \cite{PagPig}.
we first introduce the intrinsic H\"older regularity along the vector fields appearing in the
H\"ormander condition \eqref{horcon}. As it is standard in the framework of functional analysis on
homogeneous groups (cf. \cite{MR657581}), the idea is to weight the H\"older exponent in terms of
the formal degree of the vector fields, which is %is divided by the formal degree of the vector field with respect to which we
%consider the increment. Such formal degree is
equal to $1$ for $\partial_{x_1},\dots,\partial_{x_d}$ and equal to $2$ for $Y$.

\begin{definition}\label{def:intrinsic_alpha_Holder2}
%For $T>0$,
Let $\a\in\,]0,1]$, $\b\in\, ]0,2]$ and $T>0$. % and $\O$ be domain in $\rn$. %and $i=1,\dots,d$. We
We denote respectively by $C_{d}^{\a}(\mathcal{S}_{T})$ and $C_{Y}^{\b}(\mathcal{S}_{T})$ the
%spaces of functions %\red{$f\in L^{\infty}([0,T];C(\rn))$}
{set of the functions $f:\mathcal{S}_{T} \to \R$} such that the following semi-norms are finite
\begin{align}\label{e12}
 \left\|f\right\|_{C_{d}^{\a}(\mathcal{S}_{T})} &:=%\sup_{t\in[0,T]}
 \sum_{i=1}^{d}%\\ \d\neq0
 \sup_{(t,x)\in\mathcal{S}_{T}\atop h\in\R
 }\frac{\left|f (t,x +h \mathbf{e}_i )  -
 f(t,x)\right|}{|h|^{\a}}, \\
 \left\|f\right\|_{C_{Y}^{\b}(\mathcal{S}_{T})} &:=\sup_{t,s\in[0,T]\atop x\in\rn}\frac{\left|f(s,e^{(s-t) B}x) -
 f(t,x)\right|}{|t-s|^{\frac{\b}{2}}}.
\end{align}
%where we assume $f\equiv0$ on $\rnn\setminus \O$ by convention.
Here $\mathbf{e}_i$ denotes the $i$-th element of the canonical basis of $\R^N$.
%{Notice that the sum above sees the increments of $f$ only along the first $d$ components of $x$.} %We say that
%$f\in C_{d,\text{\rm loc}}^{\a}(\mathcal{S}_{T})$ (resp. $C_{Y,\text{\rm loc}}^{\b}(\mathcal{S}_{T})$) %of locally H\"older
%%continuous functions are defined in a obvious way.
%if $f\chi\in C_{d}^{\a}(\mathcal{S}_{T})$ (resp. $C_{Y}^{\b}(\mathcal{S}_{T})$) for any $\chi \in
%C_{0}^{\infty}(\mathcal{S}_{T})$.
\end{definition}

%\begin{definition}\label{def:intrinsic_alpha_Holder2}
%%For $T>0$,
%Let $\a\in\,]0,1]$, $\b\in\, ]0,2]$ and $\O$ be domain in $\rnn$. %and $i=1,\dots,d$. We
%We denote respectively by $C_{d}^{\a}(\O)$ and $C_{Y}^{\b}(\O)$ the spaces of \red{continuous}
%functions on $\O$ such that the following norms are finite
%\begin{align}\label{e12}
% \left\|f\right\|_{C_{d}^{\a}(\O)} &=\sum_{i=1}^{d}\sup_{\substack{(t,x)\in \O\\ \d\neq0}}%\\ \d\neq0
% \frac{\left|f (t,x +\delta \mathbf{e}_i )  -
% f(t,x)\right|}{|h|^{\a}}, \\
% \left\|f\right\|_{C_{Y}^{\b}(\O)} &=\sup_{\substack{(t,x)\in \O\\ \d\neq0}}
% \frac{\left|f(t+\d,e^{\d B}x) -
% f(t,x)\right|}{|h|^{\frac{\b}{2}}},
%\end{align}
%where we assume $f\equiv0$ on $\rnn\setminus \O$ by convention. Here $\mathbf{e}_i$ denotes the
%$i$-th element of the canonical basis of $\R^N$. The spaces $C_{d,\text{\rm loc}}^{\a}(\O)$ and
%$C_{Y,\text{\rm loc}}^{\b}(\O)$ of locally H\"older continuous functions are defined in a obvious
%way.
%\end{definition}
%The definition of $C_{B}^{2,\alpha}(\mathcal{S}_T)$ is given recursively in terms of
%the lower order H\"older space, namely $C_{B,\text{\rm loc}}^{0,\alpha}(\mathcal{S}_T)$ and
%$C_{B,\text{\rm loc}}^{1,\alpha}(\mathcal{S}_T)$.
Next we recall the intrinsic H\"older spaces of order 0 and 1 introduced in \cite{pagpaspig} and
\cite{MR3660883}.
{\begin{definition}\label{def:C_alpha_spaces} For $\a\in\,]0,1]$,
$C_{B}^{0,\a}(\mathcal{S}_T)$ and $C_{B}^{1,\a}(\mathcal{S}_T)$ denote, respectively, the {set of
the functions $f:\mathcal{S}_{T} \to \R$} such that the following semi-norms are finite
\begin{align}
  \|f\|_{C^{0,\a}_{B}(\mathcal{S}_T)}&:=\|f\|_{C^{\a}_{Y}(\mathcal{S}_T)}+\|f\|_{C^{\a}_{d}(\mathcal{S}_T)},\\
  \|{f}\|_{C^{1,\a}_{B}(\mathcal{S}_T)}&:=\|f\|_{C^{1+\a}_{Y}(\mathcal{S}_T)}+\sum\limits_{i=1}^{d} \|{\partial_{x_i}f}\|_{C_{B}^{0,\a}(\mathcal{S}_T)}\\
  &={\|f\|_{C^{1+\a}_{Y}(\mathcal{S}_T)}+\sum\limits_{i=1}^{d}  \big(   \|  \partial_{x_i}f  \|_{C^{\a}_{Y}(\mathcal{S}_T)}+\|  \partial_{x_i}f  \|_{C^{\a}_{d}(\mathcal{S}_T)}    \big)}  .
\end{align}
\end{definition}
}

\begin{remark}\label{rem:C0alpha_intrins}
It was shown in \cite{pagpaspig} that if $f\in C_{B}^{0,\a}(\mathcal{S}_T)$ then $f$ is $\a$-H\"older continuous w.r.t. the intrinsic norm $|(t,x)|=|t|^{\frac{1}{2}}+|x|_{B}$. % {\blue and the intrinsic translation. %in the standard Euclidean sense
%and, in particular,
In particular, $f$ enjoys some H\"older regularity also in the degenerate variables $x_{i}$ for $i>d$, namely %in all the variables, in the sense that \eqref{eq:holder_poli} holds
%locally on $\mathcal{S}_T$.
%if $f\in C_{B}^{0,\a}(\mathcal{S}_T)$, then
%\begin{equation}
%|f(t,x) - f(s,y)|  \leq c \big(   |x - y|^{\alpha}_B +|t-s|^{\frac{\alpha}{2}}   \big) , \qquad (t,x),(s,y)\in\SS_T,
%\end{equation}
%where $c>0$ is a universal constant that only depends on the matrix $B$.
%In particular, for any $t\in[0,T[$ we have
\begin{equation}
\sup_{x,y\in\R^N %\atop x\neq y
} \frac{|f(t,x) - f(t,y)| }{ |x - y|^{\alpha}_B  } \leq C \|f\|_{C^{0,\a}_{B}(\mathcal{S}_T)},
\qquad t\in[0,T[,
\end{equation}
where $C$ is a positive constant that depends only on the matrix $B$.
\end{remark}
 \begin{remark}
In \cite{pagpaspig} it was also shown that if $f\in C_{B}^{1,\a}(\mathcal{S}_T)$ then the following
\emph{intrinsic Taylor formula} holds:
\begin{equation}
\Big| f(s, y) - f(t ,x) - \sum_{i=1}^d \partial_{x_i} f(t , x ) (y - e^{(s-t) B }x)_i \Big| \leq C
\big( |s-t|^{\frac{1}{2}}+|y - e^{(s-t) B }x|_{B}  \big)^{1+\alpha}, \qquad (t,x),(s,y)\in
\mathcal{S}_T.
\end{equation}
where $C$ is a positive constant that depends only on the matrix $B$. We stress that the Taylor
``polynomial" above only contains the first derivatives of $f$ w.r.t. the first $d$ components of
$x$.
\end{remark}
To cope with the lack of regularity of the coefficients in the time-direction, the definition of
$C_{B}^{2,\alpha}(\mathcal{S}_T)$ differs from the one given in
\cite{pagpaspig}, specifically with regards to the regularity along $Y$.%: this is explained in
%more detail in Remark \ref{ra1} below.
\begin{definition}\label{def:C2_alpha_space}
For $\a\in\,]0,1]$, %we say that %Let $T>0$, $f:\mathcal{S}_T \to \R$ and $\a\in\,]0,1]$, then
$C_{B}^{2,\a}(\mathcal{S}_T)$ is %the space of functions $f$
{the set of functions $f:\mathcal{S}_{T} \to \R$} such that there exist
\begin{itemize}
  \item[i)] $\p_{x_{i}}f\in C_{B}^{1,\a}(\mathcal{S}_T)$ for $i=1,\dots, d$;
 \item[ii)] a function $f_{Y}\in L^{\infty}([0,T];C^{\alpha}_{B}(\rn))$,
 called {\it a.e.-Lie derivative of $f$}, such that
 \begin{equation}\label{ae4}
% f\big( s, e^{(s - t) B} x \big) - f(t,x) = \int_t^s  f_{Y}\big(  \tau, e^{(\tau - t) B} x  \big)
%\dd \tau, \qquad s\in\,]t,T[,
 f(s,e^{(s - t) B}x) = f(t,x) + \int_t^s  f_{Y}(\tau, e^{(\tau - t) B} x )d \tau, \qquad t,s\in[0,T].%0\le t\le
% s\le T.
\end{equation}
%for any $(t,x)\in \mathcal{S}_T$.
\end{itemize}
%Furthermore, for any $V := [0,T-\eps] \times B_n(0)$ with $\eps <T$ and $n\in\N$,
{Equivalently, $C_{B}^{2,\a}(\mathcal{S}_T)$ denotes %, respectively,
the
%spaces of functions %\red{$f\in L^{\infty}([0,T];C(\rn))$}
{set of functions $f:\mathcal{S}_{T} \to \R$} such that the following semi-norm is finite }
%We also {define the semi-norm}
\begin{align}
 \hspace{-30pt}\|{f}\|_{C^{2,\a}_{B}(\mathcal{S}_T)} & = \sum_{i=1}^{d}\|{\partial_{x_i}f}\|_{C_{B}^{1,\a}(\mathcal{S}_T)}+
  \| f_{Y} \|_{L^{\infty}([0,T];C^{\alpha}_{B}(\rn))} \\
  &  { =    \sum_{i=1}^{d}
  \|\partial_{x_i}f\|_{C^{1+\a}_{Y}(\mathcal{S}_T)}+\sum\limits_{i,j=1}^{d}  \big(   \|  \partial_{x_i x_j}f  \|_{C^{\a}_{Y}(\mathcal{S}_T)} + \|  \partial_{x_i  x_j}f  \|_{C^{\a}_{d}(\mathcal{S}_T)} \big)
  + \| f_{Y} \|_{L^{\infty}([0,T];C^{\alpha}_{B}(\rn))}         }. \label{eq:norm_C2alpha}% + %\sum_{i=1}^{d}
%\|{f_{Y}}\|_{C^{\a}_{d}(\mathcal{S}_T)} .
\end{align}
\end{definition}
{\begin{remark}\label{ra1} To have a quick comparison with the literature on the regularity for
Kolmogorov operators in \eqref{Lter}, we recall that:
\begin{itemize}
  \item[i)] the H\"older space $C^{2+\a}$ introduced in \cite{MR1475774} (and
  adopted in \cite{MR2136978}, \cite{MR2534181} to prove Schauder estimates), consists of functions
  $f$ that, together with their second order derivatives $\p_{x_{i}x_{j}}f$ in the non-degenerate directions $i,j=1,\dots,d$, are H\"older continuous w.r.t. the
  anisotropic norm \eqref{e7}. This notion is weaker than Definition \ref{def:C2_alpha_space} both
  in terms of the regularity of $\p_{x_{i}}f$ and, more importantly, in terms of the
  Lipschitz continuity of $f$ along $Y$ (cf. \eqref{ae4}) which reveals the regularizing effect of the associated evolution
  semigroup;

  \item[ii)] Definition \ref{def:C2_alpha_space} is similar in spirit to that proposed in \cite{MR1751429}, \cite{difpol} and
  \cite{pagpaspig} for the study of Kolmogorov operators with H\"older %(jointly in space and time)
  coefficients: according to their definition if $f\in C^{2,\alpha}$ then $Yf$ exists and belongs $C^{0,\alpha}_B$. %In particular, it also belongs to $C^{\a}_{Y}( \mathcal{S}_T)$:
%in other words, $f_{Y}$ is also locally $\alpha/2$-H\"older continuous along $Y$.
%
%  \item[i)] According to the definition of $C_{B}^{2,\alpha}$ given in \cite{pagpaspig} (basically, not exactly, equivalent to \cite{Manfredini}
%  and \cite{DiFPolidoro}), if  {\blue
%which is the standard regularity that one may expect in case the coefficients of $\A$ are H\"older continuous in both space and time.
%Even though it is not explicitly stated in the previous literature,
This is the regularity that the fundamental solution enjoys in case the coefficients of $\A$ are
H\"older continuous in both space and time. By contrast, if $f\in C_{B}^{2,\alpha}$ in the sense
of Definition \ref{def:C2_alpha_space} then $f$ is generally at most Lipschitz continuous along
$Y$: this is the optimal result one can prove without assuming further regularity of the
coefficients in the time variable other than measurability.
\end{itemize}
\end{remark}

%Critically, this difference will have an impact on the estimate of the remainder of the second
%order intrinsic Taylor polynomial and \red{in the parametrix construction.}

\vspace{5pt} The rest of the paper is structured as follows. Section \ref{sec:parametrix} contains
the construction of the fundamental solution by means of the parametrix method: in particular,
Section \ref{sec:proof_theorem_existence} includes the proof of Theorem \ref{main}. In Section
\ref{regolaritasoluzione} we prove the regularity estimates of the fundamental solution, in
particular Theorem \ref{ta1}. In Section \ref{sec:Cau} we state some results for the Cauchy
problem for $\A+Y$. The appendices contain the Gaussian and potential estimates that are employed
in the proofs.

For reader's convenience, we recall that we shall always denote by $\mathcal{S}_{T}$ the strip
$]0,T[\times\rn$; also, in the following table we collect the notations used for the main
functional spaces:
\begin{center}
\begin{tabular}{|c|c|c|}
 \hline
 Notation & Functional space & Reference \\
 \hline
% Bounded continuous functions & $C_b(\R^N)$ & Assumption \ref{coer} \\
 $C^\a_B$ & Anisotropic H\"older spaces on $\rn$ & Def. \ref{def:holder_coeff}\\
% Coefficients' regularity & $L^\infty([0,T_0];C^\a_B(\R^N))$ & Definition \ref{def:holder_coeff}\\
 %Sufficient regularity of the strong Lie solutions
% Strong Lie solution spaces & $L_{\text{\rm loc}}^{1}([0,T[;C_{b}(\rn))$ & Definition \ref{solint}\\
 $C_{d}^{\a}, C_{Y}^{\a}$ & Lie H\"older spaces on $\mathcal{S}_{T}$ & Def. \ref{def:intrinsic_alpha_Holder2}\\
 $C_{B}^{k,\a},\ k=0,1,2$ & Intrinsic H\"older spaces on $\mathcal{S}_{T}$ & Def. \ref{def:C_alpha_spaces}, \ref{def:C2_alpha_space}\\
% Intrinsic H\"older spaces of order $2$& $C_{B}^{2,\a}(\mathcal{S}_T)$ & Definition \ref{def:C2_alpha_space}\\
 \hline
\end{tabular}
\end{center}
}
\section{Parametrix construction}\label{sec:parametrix}
%Throughout the rest of the paper, Assumptions \ref{coer}, \ref{ass:hypo} and \ref{ass:regularity} are satisfied.
Let Assumptions \ref{coer}, \ref{ass:hypo} and \ref{ass:regularity} be satisfied. The first step
of the parametrix method is to set a kernel $\param=\param(t,x;T,y)$ that serves as proxy for the
fundamental solution, called \emph{parametrix}. We denote by $\Ac^{(s,v)}$ the operator obtained
by freezing the second-order coefficients of $\Ac$ along the integral curve of the vector field
$Y$ passing through $(s,v)\in\mathcal{S}_{\TT}$ and neglecting the lower order terms. Namely we
consider the operator
%\begin{equation}\label{Lfrozen}
%\L^{(s,v)}:=
%\underbrace{\frac{1}{2}\sum_{i,j=1}^{d}
%a_{ij}(t,e^{(t-s)B}v)\p_{x_i x_j}}_{=:\A^{(s,v)}}+ Y,
%%\< B x,\nabla\>+\p_t,
%\qquad (t,x) \in \mathcal{S}_{\TT}. %\mathcal{S}_{\TT}.
%\end{equation}
\begin{equation}\label{Lfrozen}
\Ac^{(s,v)}:= \frac{1}{2}\sum_{i,j=1}^{d} a_{ij}(t,e^{(t-s)B}v)\p_{x_i x_j},
%\< B x,\nabla\>+\p_t,
\qquad (t,x) \in \mathcal{S}_{\TT}. %\mathcal{S}_{\TT}.
\end{equation}
One can directly prove %{\blue (forse dovremmo dimostrarlo?)}
that the %{\blue (non \`e unica)}
fundamental solution of
\begin{equation}\label{eq:L_frozen}
%\L^{(s,v)} :=
\Ac^{(s,v)} + Y,
\end{equation}
%$\L^{(s,v)}$
in the sense of Definition \ref{fund}, is given by
\begin{equation}\label{eq:param}
\G^{(s,v)}(t,x;T,y)=\gg(\C ^{(s,v)}(t,T),y-e^{(T-t)B}x),\qquad
%x,y\in\R^N, \quad 0\leq t<T<\TT,
(T,y)\in \mathcal{S}_{\TT},\ (t,x)\in \mathcal{S}_T,
%\quad (t,x) \in \mathcal{S}_T,
\end{equation}
%for $(t,x),(T,y)\in\mathcal{S}_{\TT}, t<T$,
where
\begin{equation}
\gg(\mathcal{C},z):=
\frac{1}{\sqrt{(2\pi)^{N}\det \mathcal{C}}}e^{-\frac{1}{2}\<\mathcal{C}^{-1}z,z\>} %\qquad \theta\in\mathcal{M}^{N\times N}, \quad z\in\R^N,
\end{equation}
is the Gaussian kernel on $\rn$ and
\begin{align}\label{cao}
\C^{(s,v)}(t,T)&:= \int_{t}^{T} e^{(T-\t) B} A^{(s,{v})} (\t) e^{(T-\t) B^*}d\t , \\ \label{aao}
%A^{(s,v)}(\t)&:=\begin{pmatrix} A_0(\t,e^{(\t-s)B}v)& 0_{{d} \times (N-d)} \\ 0_{(N-d) \times {d}}
%& 0_{(N-d) \times (N-d)}\end{pmatrix},\qquad A_{0}=\(a_{ij}\)_{i,j=1,\dots,d}.
A^{(s,v)}(\t)&:=\begin{pmatrix} A_0(\t,e^{(\t-s)B}v)& 0 \\ 0 &
0\end{pmatrix},\qquad A_{0}=\(a_{ij}\)_{i,j=1,\dots,d}.
\end{align}
\begin{remark}
%As noticed in \cite{brampol},
Clearly $\G^{(s,v)}(t,x;T,y)$ is of class $C^{\infty}$ as a function of $x$ and only absolutely
continuous along the integral curves of $Y$ as a function of $(t,x)$.
%, for a fixed $(T,y)$, the function $(t,x)\mapsto \G^{(s,v)}(t,x;T,y)$ %is not
%%differentiable with respect to $t$ everywhere on $\mathcal{S}_T$. However it
%is absolutely continuous along the integral curves of $Y$.
\end{remark}

\begin{remark}
In the particular case of $A_{0}\equiv \d I_{d}$ for some $\d>0$, where $I_{d}$ is the $(d\times
d)$-identity matrix,
%\begin{equation}
%%a_{ij}(t,x) \equiv \delta >0, \qquad i,j=1,\dots, d,
% A_{0}=\d I_{d}
%\end{equation}
the Kolmogorov operator $\Ac^{(s,v)} + Y$ reads as in \eqref{eq:kolm_const} and its fundamental
solution reduces to
\begin{equation}\label{gaussian}
\G^{\d}(t,x;T,y):=\gg(\d\mathcal{C}(T-t),y-e^{(T-t)B}x),
%\\&=\frac{\(2\d\pi\)^{-\frac{N}{2}}}{\sqrt{\det \mathcal{C}(T-t)}}
%\exp\(-\frac{1}{2\d}\< \mathcal{C}(T-t)^{-1}
%\(y-e^{(T-t)B}x\),\(y-e^{(T-t)B}x\)\>\),
\end{equation}
with
\begin{equation}\label{eq:Ct}
\mathcal{C}(t)=
\int_{0}^{t}
e^{(t-\t) B} %A
\begin{pmatrix} I_{d} & 0 \\
0 & 0\end{pmatrix} e^{(t-\t) B^*}d\t.
\end{equation}
\end{remark}
{Proceeding as in \cite{menozzi} and \cite{MR4355925},} we define the parametrix
function $\param(t,x;T,y)$ as
\begin{equation}\label{eq:def_parametrix}
 \param(t,x;T,y):=\G^{(T,y)}(t,x;T,y),\qquad  (T,y)\in \mathcal{S}_{\TT}, \ (t,x)\in \mathcal{S}_T,
\end{equation}
and we refer to it as to the \emph{time-dependent parametrix} in order to emphasize the fact that it is obtained by freezing only the space variable of the coefficients of $\Ac$.
\begin{remark}\label{rem:eq_param}
Since $\G^{(s,v)}$ is the fundamental solution of $\Ac^{(s,v)} + Y$, we have
\begin{equation}\label{eq:int_eq_frozen}
 (\A^{(T,y)}+Y) \param(\cdot,\cdot;T,y)   = 0 \quad \text{on } \mathcal{S}_T,
\end{equation}
in the sense of Definition \ref{solint}, for any $(T,y)\in\mathcal{S}_{\TT}$.
\end{remark}
%for any $(t,x),(T,y)\in \mathcal{S}_{\TT}, t<T$.
\begin{remark}\label{rem:classical_parametrix}
In \cite{difpas}, where the variable coefficients of $\Ac$ are assumed intrinsically H\"older continuous in space and time, %in the sense of \eqref{eq:holder_poli},
 %function $\G^{(s,v)}$
%the kernel that defines
the parametrix is defined as the fundamental solution of the operator %$\Lc^{(s,v)}$ as
%in \eqref{eq:L_frozen}, with $\Ac^{(s,v)}$ being the operator with constant
%coefficients,
obtained by freezing the second order coefficients of $\Ac$ in both time and space
variables, i.e.
\begin{equation}\label{Lfrozen_bis}
%\Ac^{(s,v)}:=
 \frac{1}{2}\sum_{i,j=1}^{d} a_{ij}(s,v)\p_{x_i x_j} + Y.
%%\< B x,\nabla\>+\p_t,
%\qquad (t,x) \in \mathcal{S}_{\TT}. %\mathcal{S}_{\TT}.
\end{equation}
As we shall see below, the choice of freezing the coefficients only in the space variable, along
the integral curve of $Y$ as in \eqref{Lfrozen}, is necessary in order to deal with the lack of regularity in the time
variable.
\end{remark}
Once the parametrix function is defined, the parametrix construction prescribes that a fundamental
solution
%According to the parametrix technique, we want to prove that a fundamental solution of
of $\mathcal{A}+Y$ is sought in the form
%takes the form
\begin{equation}\label{FS}
\sol(t,x;T,y)=
\param(t,x;T,y)+
\int_t^T\int_{\R^N} \param(t,x;\t,\y)\phi(\t,\y;T,y)d\y d\t,
\end{equation}
where $\phi$ is an unknown function. We now perform some heuristic computations that will lead to
a fixed-point equation for $\phi$. Assuming that $\sol(t,x;T,y)$ in \eqref{FS} is a fundamental
solution of $\mathcal{A}+Y$, we obtain
%that can be written as in \eqref{FS},
%we can formally apply the operator $(\mathcal{A}+Y)$ and find that
\begin{equation}
0=(\mathcal{A}+Y) \sol(t,x;T,y)= (\mathcal{A}+Y)\param(t,x;T,y)+(\mathcal{A}+Y)\int_t^T\int_{\R^N}
\param(t,x;\t,\y)\phi(\t,\y;T,y) d\y d\t.
\end{equation}
Furthermore, by formally differentiating and employing $p(t,x;t, \cdot)= \delta_x$ we also have
%Assuming that the second addend on the right hand side of the equation
%is differentiable under the integral sign, we get
\begin{equation}
\begin{aligned}
&(\mathcal{A}+Y)\int_t^T\int_{\R^N} \param(t,x;\t,\y)\phi(\t,\y;T,y) d\y d\t%\\
=%&
\int_t^T\int_{\R^N} (\mathcal{A}+Y)\param(t,x;\t,\y)\phi(\t,\y;T,y) d\y d\t - \phi(t,x;T,y).
%-\underbrace{\int_{\R^N} \param(t,x;t,\y)\phi(t,\y;T,y) d\y}_{\approx\phi(t,x;T,y)}.
\end{aligned}
\end{equation}
Therefore, $\phi(t,x;T,y)$ must solve the Volterra integral equation
\begin{equation}\label{volt}
\phi(t,x;T,y)= (\mathcal{A}+Y)\param(t,x;T,y)+\int_t^T\int_{\R^N}
(\mathcal{A}+Y)\param(t,x;\t,\y)\phi(\t,\y;T,y) d\y d\t.
\end{equation}
Now, owing to Remark \ref{rem:eq_param}, equation \eqref{volt} can be written as
\begin{equation}\label{volt_bis}
\phi(t,x;T,y)= (\A-\A^{(T,y)}) \param(t,x;T,y)+\int_t^T\int_{\R^N} (\A-\A^{(\tau,\y)})
\param(t,x;\t,\y)\phi(\t,\y;T,y) d\y d\t,
\end{equation}
whose solution can be determined by an iterative procedure, which leads to the series representation
\begin{equation}\label{serie}
\phi(t,x;T,y)=\sum\limits_{k\geq1}%((\mathcal{A}+Y)\G)
\phi_k(t,x;T,y)
\end{equation}
where %, reminding the notation in \eqref{Lparam},
\begin{equation} \label{termini}
\begin{cases}
%((\mathcal{A}+Y)\param)
\phi_1(t,x;T,y):= \(\A-\A^{(T,y)}\) \param(t,x;T,y),\\
\displaystyle%((\mathcal{A}+Y)\param)
\phi_{k+1}(t,x;T,y):=
\int_t^T\int_{\R^N} (\A-\A^{(\tau,\y)}) \param(t,x;\t,\y) %((\mathcal{A}+Y)\param)
\phi_k(\t,\y;T,y) d\y d\t, \qquad k\in\N.
\end{cases}
\end{equation}
%Since $(\mathcal{A}+Y)^{(s,v)}=\A^{(s,v)}+Y$, $(\mathcal{A}+Y)=\A+Y$ and $(\mathcal{A}+Y)^{(s,v)}\G^{(s,v)}=0$
%(in the sense of Definition \ref{solint}), we define
%\begin{equation}\label{Lparam}
%(\mathcal{A}+Y)\G^{(s,v)}(t,x;T,y):=\(\A-\A^{(s,v)}\)\G^{(s,v)}(t,x;T,y),
%\end{equation}
%for any $(s,v),(t,x),(T,y)\in \mathcal{S}_{\TT}$ with  $t<T$.
%The previous arguments are made rigorous by the following proposition:
%SERIES
In order to make the previous arguments rigorous one has to prove that: %to proceed in two steps:
\begin{itemize}
\item the series defined by \eqref{serie}-\eqref{termini} is uniformly convergent on $\mathcal{S}_T$. At this stage one also obtains a uniform upper bound and a H\"older estimate for $\phi$; %which are necessary for the next step;
\item $p$ defined by \eqref{FS} is actually a fundamental solution of $\mathcal{A}+Y$.
In this step one also establishes the Gaussian estimates on $\sol$ and its derivatives that appear
in Theorem \ref{main}.
\end{itemize}

\subsection{Convergence of the series and estimates on $\phi$}
%We have the following
\begin{proposition}\label{proposition_serie}
For every $(T,y)\in \mathcal{S}_{\TT}$ the series in \eqref{serie} converges uniformly in
$(t,x)\in\mathcal{S}_T$ and
%Moreover
the function $\phi=\phi(t,x;T,y)$ %is a continuous {\blue (non credo)} function for $(t,x),(T,y)\in\mathcal{S}_{\TT}, t<T$,
solves the integral equation \eqref{volt_bis} on $\mathcal{S}_T$. Furthermore, %and satisfies the following estimates:
for every $\e>0$ and $0<\d<\a$, there exists a positive constant $C$, only dependent on $\TT, \m,
B, \d,\a, \e$ and the $\a$-H\"older norms of the coefficients, %{\blue and on the
%$L^{\infty}([0,\TT];C^{\alpha}(\R^N))$-norms of the coefficients $a_{ij}$, $a_i$ and $a$,}
such that
\begin{align}
\left|\phi(t,x;T,y)\right|&\leq\frac{C}{(T-t)^{1-\frac{\a}{2}}}\G^{\m+\e}(t,x;T,y),\label{eq:stimephi}\\
\left|\phi(t,x;T,y)-\phi(t,v;T,y)\right|&\leq\frac{C|x-v|^{\a-\d}_B}{(T-t)^{1-\frac{\d}{2}}}\(\G^{\m+\e}(t,x;T,y)+\G^{\m+\e}(t,v;T,y)\),\label{eq:stimephihold}
%\left|\phi(t,x;T,y)-\phi(t,x';T,y)\right|&\leq\frac{C|x-x'|^{\frac{\a}{2}}}{(T-t)^{1-\frac{\a}{4}}}\(\G^{\m+\e}(t,x;T,y)+\G^{\m+\e}(t,x';T,y)\),
\end{align}
for any $(T,y)\in\mathcal{S}_{\TT}$ and $(t,x),(t,v)\in\mathcal{S}_{T}$.
%
%
%$0\leq t< T\le \TT$ and $x,y,v\in\R^N$.
%$(t,x),(T,y)\in\mathcal{S}_{\TT}, t<T$ and $x'\in\R^N$.
\end{proposition}
To avoid repeating the arguments already used in \cite{difpas}, we limit ourself to highlighting the parts of the proof that differ significantly from the classical case.
\begin{proof}
We first prove that there exists a positive $\kappa$ %, only dependent on $N, \TT, \m, B, \d, \e$
%{\blue and on the $L^{\infty}\big( [0,\TT] :   C^{\alpha}(\R^N)\big)$-norms of the coefficients
%$a_{ij}$, $a_i$ and $a$},
such that
\begin{equation}\label{eq:prelim_estimate}
 | (\A-\A^{(T,y)} ) \param(t,x;T,y) |  \leq   \frac{
\kappa}{(T-t)^{1-\alpha/2}}\G^{\m+\e}(t,x;T,y), \qquad (T,y)\in\mathcal{S}_{\TT},\
(t,x)\in\mathcal{S}_{T}.
%0\leq t<T\le
%\TT, \ x,y\in\R^N.
\end{equation}
Assume for simplicity that $a_i,a \equiv 0$, the general case being a straightforward extension. By definition \eqref{Lfrozen} we have
\begin{align}\label{eq:estimate_increment_coeff}
 |(\A-\A^{(T,y)}) \param(t,x;T,y)| & \leq \frac{1}{2}\sum_{i,j=1}^{d}
 | a_{ij}(t,x) - a_{ij}(t,e^{-(T-t)B}y)| \times | \p_{x_i x_j} \param(t,x;T,y)|\le
 \intertext{(by the H\"older regularity of $a_{ij}$ and the Gaussian estimate \eqref{derparam})}
 & \leq \kappa  \frac{| x - e^{-(T-t)B}y |^{\alpha}_{B}}{T-t}\,\G^{\m+\e/2}(t,x;T,y). \label{eq:ste101}
 \end{align}
The estimate {\eqref{polipar}} then yields \eqref{eq:prelim_estimate}.
% & \leq     \frac{ C}{(T-t)^{1-\alpha/2}}\,\G^{\m+\e}(t,x;T,y),
%\end{align}
%for any $0\leq t<T<\TT$ and $x,y\in\R^N$. By definition, this implies

For any $(T,y)\in\, \mathcal{S}_{\TT}$ and $(t,x)\in \mathcal{S}_T$,  \eqref{termini} and
\eqref{eq:prelim_estimate} imply
%$0\leq t<T<\TT$ and $x,,y\in\R^N$$
\begin{equation}
| \phi_1(t,x;T,y)  |  \leq  \frac{ \kappa }{(T-t)^{1-\alpha/2}}\,\G^{\m+\e}(t,x;T,y)
\end{equation}
and
\begin{align}
| \phi_{2}(t,x;T,y) | & \leq  \int_t^T\int_{\R^N} \left|\left(\A-\A^{(\tau,\y)}\right) \param(t,x;\t,\y)\right| %(\L\param)
\times | \phi_1(\t,\y;T,y)| d\y d\t \\ & \leq \kappa^2 \int_t^T  \frac{ 1}{(\tau-t)^{1-\alpha/2}}
\frac{ 1}{(T-\tau)^{1-\alpha/2}}   \int_{\R^N} \G^{\m+\e}(t,x;\t,\y)\G^{\m+\e}(\t,\y;T,y)  d\y
d\t= \intertext{(by the Chapman-Kolmogorov identity and solving the integral in $d\t$)} & =
\kappa^2\, \frac{ \G^2_{\text{Euler}}\left(\frac{\a}{2}\right) }{ (T-t)^{1 - \alpha
}\G_{\text{Euler}}( \alpha ) } \G^{\m+\e}(t,x;T,y).
\end{align}
Proceeding by induction, it is straightforward to verify that
\begin{equation}
 | \phi_{n}(t,x;T,y) |  \leq \kappa^{n}\frac{\G_{\text{Euler}}^{n}\left(\frac{\a}{2}\right)}{(T-t)^{1-\frac{\a n}{2}}\G_{\text{Euler}}
 (\frac{\a n}{2}) }\G^{\m+\e}(t,x;T,y), \qquad n\in \N.
\end{equation}
This proves the uniform convergence of the series on $\mathcal{S}_T$, which in turn implies that
$\phi$ satisfies \eqref{volt_bis}, as well as the estimate \eqref{eq:stimephi}.

The proof of \eqref{eq:stimephihold} is a technical modification of the arguments in \cite[Lemma
6.1]{difpas}, which is necessary to account for the different parametrix function. For sake of
brevity, we leave the details to the reader.
\end{proof}

\begin{remark}
The proof above is particularly informative to understand the choice of the parametrix function in
relation to the lack of regularity of the coefficients with respect to the time variable. In
particular, in passing from \eqref{eq:ste101} to \eqref{eq:prelim_estimate}, we take advantage of
the increment $| x - e^{-(T-t)B}y |_B^{\alpha}$ in order to recover the integrability of the
singularity in time. In the classical case, namely when the coefficient $a_{ij}$ is also H\"older
continuous in time, the parametrix function is obtained by freezing the variable coefficients in
both space and time (see Remark \ref{rem:classical_parametrix}). In
\eqref{eq:estimate_increment_coeff}, this choice leads to increments of type
 \begin{equation}
 \big| a_{ij}(t,x) - a_{ij}( T, y )  \big|,
\end{equation}
which is clearly not helpful if $a_{ij}$ does not exhibit any regularity in time.

Furthermore, note that the coefficients have to be frozen in the space variable along the integral
curve of $Y$: freezing the coefficients at a fixed point $y$ would yield an increment of type
%\begin{equation}
 $ | x - y  |_B^{\alpha} $
%\end{equation}
in \eqref{eq:ste101}, which does not allow to employ the Gaussian estimates in {\eqref{polipar}}
to control the singularity.
\end{remark}

\subsection{Proof that $p$ is a fundamental solution {and Gaussian bounds}}\label{sec:proof_theorem_existence}
We now prove the first part of Theorem \ref{main}, concerning the existence of the fundamental
solution of $\Lc$. This is achieved by proving that the candidate solution $\sol=\sol(t,x;T,y)$
defined through \eqref{FS} satisfies points i) and ii) of Definition \ref{fund}. The innovative
part of the proof consists in showing point i), which is $\sol(\cdot,\cdot;T,y)$ solves the
equation
\begin{equation}\label{eq:Yu_eq_A}
\A u +Y  u=0\quad \text{on $\mathcal{S}_T$}
\end{equation}
%$Y u=-\A u$ on $\mathcal{S}_T$
in the integral sense of Definition \ref{solint}. Once more, we provide the details of the parts that significantly depart from the classical case.

For any $(T,y)\in \mathcal{S}_{\TT}$, let us rewrite $\sol(t,x;T,y)$ as
\begin{equation}
\sol(t,x;T,y) = \param(t,x;T,y) + \Phi(t,x;T,y), \qquad (t,x)\in \mathcal{S}_T,
\end{equation}
%First, we prove that the parametrix $\param$ satisfies a final datum condition for the Cauchy problem:
where we set
%
%We use the notation
\begin{equation}\label{eq:def_Phi}
\Phi(t,x;T,y):=\int_t^T\int_{\R^N} \param(t,x;\t,\y)\phi(\t,\y;T,y)d\y d\t.
\end{equation}
The strategy of the proof is to first show that $\sol$ possesses the regularity required in order
to qualify as a fundamental solution, and then to check that it actually solves equation
\eqref{eq:Yu_eq_A}. As pointed out in Remark \ref{eq:int_eq_frozen}, the parametrix
$\param=\param(t,x;T,y)$ is an integral solution to \eqref{eq:int_eq_frozen}. In particular, it is
smooth in the variable $x$ and absolutely continuous along $Y$. As for $\Phi=\Phi(t,x;T,y)$, the
next result shows that it is twice differentiable w.r.t. $x_1,\dots,x_{d}$ and states some
Gaussian bounds on the derivatives.
\begin{proposition}\label{lem:deriv_theta}
For any $(T,y)\in\mathcal{S}_{\TT}$, $(t,x)\in \mathcal{S}_T$ and $i,j = 1,\dots,d$, there exist
\begin{align}
 \partial_{x_i } \Phi(t,x;T,y) & = \int_t^T\int_{\R^N} \partial_{x_i} \param(t,x;\t,\y)\phi(\t,\y;T,y)d\y d\t,\\
 \partial_{x_i x_j } \Phi(t,x;T,y) & = \int_t^T\int_{\R^N} \partial_{x_{i}x_{j}} \param(t,x;\t,\y)\phi(\t,\y;T,y)d\y d\t,
\end{align}
and, for any $\eps>0$ we have
\begin{align}
\left|{\Phi}(t,x;T,y)\right| &\leq C (T-t)^{\frac{\alpha}{2}} \G^{\m+\e}(t,x;T,y), \\
\left|\p_{x_i}{\Phi}(t,x;T,y)\right| &\leq \frac{C}{(T-t)^{\frac{1-\alpha}{2}}}
\G^{\m+\e}(t,x;T,y),
\\ \left|\p_{x_i x_j}{\Phi}(t,x;T,y)\right| &\leq \frac{C}{(T-t)^{\frac{2-\alpha}{2}}} \G^{\m+\e}(t,x;T,y),
\end{align}
where $C$ denotes a positive constant, only dependent on $\TT, \m, B,\a, \e$ and the $\a$-H\"older
norms of the coefficients.
\end{proposition}
\begin{proof}
{By the definition of $\Phi$ in \eqref{eq:def_Phi} we have
\begin{equation}
\Phi(t,x;T,y) =  \int_t^T J(t,x;\t;T,y) d\t  ,
\end{equation}
with $J$ defined as in \eqref{eq:def_J}. The potential estimates {of Proposition \ref{lem:stime_pot}} %\eqref{eq:pot_est}-\eqref{der1pot}-\eqref{der2pot} of Remark \ref{rem:potential_etimates},
upon integrating in $\tau$, yield the result.}
\end{proof}
The following result shows that $\Phi(\cdot,\cdot;T,y)$ is also Lipschitz continuous along the
integral curves of $Y$.
%We have that $\Phi$ satisfies an integral version of the Leibniz formula:
% and refer to Appendix \ref{appendix:potentialestimates} for the potential estimates on which
\begin{lemma}\label{leib}
For every $(T,y)\in \mathcal{S}_{\TT}$ and $(t,x)\in \mathcal{S}_T$, we have
%$(t,x),(T,y)\in \mathcal{S}_{\TT}, t<T$, we have that
\begin{equation}\label{leibeq}
\Phi\big(s,e^{(s-t)B}x;T,y\big)-\Phi(t,x;T,y)=-\int_t^s F(\t,x;T,y)d\t, \qquad s\in[t,T[,
\end{equation}
where
\begin{equation}\label{eq:def_F}
F(\t,x;T,y):= \int_\t^T\int_{\R^N} \A^{(r,\y)}\param\big(\t,e^{(\t-t)B}x;r,\y\big)\phi(r,\y;T,y)d\y dr + \phi\big(\t,e^{(\t-t)B}x;T,y\big).
\end{equation}
\end{lemma}
\begin{proof}
For any $s\in[t,T[ $ one can write
\begin{align}
\begin{aligned}
\Phi\big(s,e^{(s-t)B}x;T,y\big)-\Phi(t,x;T,y)=&
\underbrace{\int_{s}^{T}\int_{\R^N} \(\param\big(s,e^{(s-t)B}x;r,\y\big)-\param(t,x;r,\y)\)\phi(r,\y;T,y)d\y dr}
_{\eqqcolon G(t,x)}\\
&-\underbrace{\int_{t}^{s}\int_{\R^N} \param(t,x;r,\y)\phi(r,\y;T,y)  d\y dr}_{\eqqcolon H(t,x)}.
\end{aligned}
\end{align}
%\begin{align}
%\begin{aligned}
%\Phi(s,e^{(s-t)B}x;T,y)-\Phi(t,x;T,y)=&
%\int_{s}^{T}\int_{\R^N} \phi(r,\y;T,y)\(\param(s,e^{(s-t)B}x;r,\y)-\param(s,e^{(s-t)B}x;r,\y)\)d\y dr\\
%&+\int_{t}^{s}\int_{\R^N} \phi(r,\y;T,y) \param(s,e^{(s-t)B}x;r,\y) d\y dr\\
%\intertext{(poich\'e $\param(s,y;\cdot,x)$ \`e soluzione fondamentale di $\L^{(T,y)}$)}
%=&\underbrace{\int_{s}^{T}\int_{\R^N} \phi(r,\y;T,y)\(\int_t^{s}\A^{(r,\y)}\param(\t,e^{(\t-t)B}x;r,\y)d\t\)d\y dr}_{\eqqcolon G(t,x)}\\
%&+\underbrace{\int_{t}^{s}\int_{\R^N} \phi(r,\y;T,y) \param(s,e^{(s-t)B}x;r,\y) d\y dr}_{\eqqcolon H(t,x)}.
%\end{aligned}
%\end{align}
%
First, we study the term $G(t,x)$.
%Since $\param(\cdot,\cdot;T,y)$ is integral solution of $\L^{(T,y)}u=0$
Remark \ref{rem:eq_param} yields %we have
\begin{equation}
G(t,x)={\int_{s}^{T}\int_{\R^N} \(\int_t^{s}-\A^{(r,\y)}\param\big(\t,e^{(\t-t)B}x;r,\y\big)d\t\)\phi(r,\y;T,y)d\y dr}.
\end{equation}
By \eqref{derparam} and Assumption \ref{ass:regularity}, %the boundedness of the coefficients of $\A$,
%we find that
for every $\e>0$ we have %there exists %a positive constant $C$ such that
\begin{equation}
\left| \A^{(r,\y)}\param(\t,e^{(\t-t)B}x;r,\y) \right| \leq \frac{C}{r-\t}
\G^{\m+\e}(\t,e^{(\t-t)B}x;r,\y), \qquad t< \t<s< r< T.%\t\in[t,s[.
\end{equation}
%for any $\t\in[t,s[$.
Therefore, considering also \eqref{eq:stimephi}, for any $r\in\,]s,T]$, the function
\begin{equation}
(\t,\y)\mapsto | \A^{(r,\y)}\param(\t,e^{(\t-t)B}x;r,\y)  \phi(r,\y;T,y)  |
\end{equation}
%$(\t,\y)\mapsto | \A^{(r,\y)}\param(\t,e^{(\t-t)B}x;r,\y) | $
is integrable on $[t,s]\times\R^N$. Thus Fubini's theorem yields %for any $r\in\,]s,T]$ and $x\in\R^N$.
%Therefore, by Fubini theorem, we have
\begin{equation}
\begin{aligned}
&\int_{\R^N}\(\int_t^{s}\A^{(r,\y)}\param\big(\t,e^{(\t-t)B}x;r,\y\big)  d\t \) \phi(r,\y;T,y) d\y
\\ =&\int_t^{s} \int_{\R^N} \A^{(r,\y)}\param\big(\t,e^{(\t-t)B}x;r,\y\big)\phi(r,\y;T,y)d\y
d\t.
\end{aligned}
\end{equation}
Moreover, %for equation
by the potential estimate \eqref{derpot} with $\delta=\frac{\a}{2}$, for any $\eps>0$
%there exists a positive constant $C$ such that
we have
\begin{equation}\label{stima1}
\begin{aligned}
\left| \int_{\R^N}\A^{(r,\y)}\param\big(\t,e^{(\t-t)B}x;r,\y\big)\phi(r,\y;T,y)d\y \right| \leq
{\frac{C}{(T-r)^{1-\frac{\a}{4}}(r-\tau)^{1-\frac{\a}{4}}}
}\G^{\m+\e}\big(\t,e^{(\t-t)B}x;T,y\big).
\end{aligned}
\end{equation}
As the right-hand side term is integrable over $[t,s]\times [s,T]$ as a function of $(\t,r)$, we
can apply once more Fubini's theorem %and obtain
%%Therefore, for every $x,y\in\R^N$, the function $(\t,r)\mapsto\int_{\R^N}\param(\t,e^{(\t-t)B}x;r,\y)\phi(r,\y;T,y)\A^{(r,\y)}d\y$
%%is integrable on $[t,s]\times [s,T]$, and by Fubini's theorem
%\begin{equation}
%\begin{aligned}
%&\int_{s}^{T} \int_t^{s} \(\int_{\R^N} \A^{(r,\y)}\param\big(\t,e^{(\t-t)B}x;r,\y\big)\phi(r,\y;T,y)d\y \) d\t dr\\
%=&
%\int_t^{s} \int_{s}^{T}  \(\int_{\R^N} \A^{(r,\y)}\param\big(\t,e^{(\t-t)B}x;r,\y\big)\phi(r,\y;T,y)d\y \) dr d\t.
%\end{aligned}
%\end{equation}
%We have proved that %for $0\leq t<s<T$ and $x,y\in\R^N$ we have
to conclude that
\begin{equation}\label{eq:ste103}
G(t,x) =-\int_t^{s} \int_{s}^{T} \int_{\R^N}
\A^{(r,\y)}\param\big(\t,e^{(\t-t)B}x;r,\y\big)\phi(r,\y;T,y)d\y dr d\t.
%\Phi(s,e^{(s-t)B}x;T,y)-\Phi(t,x;T,y)
%=&\int_t^{s} \int_{s}^{T}  \(\int_{\R^N} \phi(r,\y;T,y)\A^{(r,\y)}\param(\t,e^{(\t-t)B}x;r,\y)d\y \) dr d\t\\
%&+\int_{t}^{s}\int_{\R^N} \phi(t_0,x_0;\t,y) \param(\t,y;s,x) d\y d\t.
\end{equation}

Let us consider $H(t,x)$. For every $n\in\N$, we define $\e_n(r):=\frac{1}{n}(r-t)$.
Note that, for every $r\in\,]t,s[$ %and $n>1$,
we have $r-\e_n(r)\geq t$.
Hence
\begin{align}
\begin{aligned}
H(t,x) =&\underbrace{ \int_{t}^{s}\int_{\R^N}
\param\big(r-\e_n(r),e^{(r-\e_n(r)-t)B}x;r,\y\big)\phi(r,\y;T,y)  d\y dr }_{=:\tilde H_n(t,x)}\\
&-\underbrace{\int_{t}^{s}\int_{\R^N}
\(\param\big(r-\e_n(r),e^{(r-\e_n(r)-t)B}x;r,\y\big)-\param(t,x;r,\y)\) \phi(r,\y;T,y) d\y dr} _{
=:H_n(t,x)}.
\end{aligned}
\end{align}
%Since $\param(\cdot,\cdot;r,\y)$ is integral solution of $\L^{(r,\y)}u=0$, we have
Once more, Remark \ref{rem:eq_param} yields
\begin{align}
H_{n}(t,x)&= \int_{t}^{s}\int_{\R^N}\bigg(
\int_{t}^{r-\e_n(r)}\A^{(r,\y)}\param\big(\t,e^{(\t-t)B}x;r,\y\big)d\t \bigg)  \phi(r,\y;T,y)d\y
dr
%\end{equation}
%Proceeding as before, we get
\intertext{(applying Fubini's theorem as above)}
%\begin{align}
%H_{n}(t,x)
&= \int_{t}^{s} \int_{t}^{r-\e_n(r)} \int_{\R^N}\A^{(r,\y)}\param\big(\t,e^{(\t-t)B}x;r,\y\big)
\phi(r,\y;T,y)d\y d\t dr \intertext{(setting $\d_n(\t)=\frac{\t-t}{n-1}$ and applying Fubini's
theorem again)}
%, so that $\{(r,\t)|t<r<s,t<\t<r-\e_n(r)\}=\{(r,\t)|t<\t<s-\e_n(s),\t+\d_n(\t)<r<s\}$)}
&= \int_{t}^{s-\e_n(s)}\int_{\t+\d_n(\t)}^{s}\int_{\R^N}
\A^{(r,\y)}\param\big(\t,e^{(\t-t)B}x;r,\y\big) \phi(r,\y;T,y)d\y dr d\t \intertext{(by
\eqref{stima1} and applying Lebesgue's dominated convergence theorem)} &\xrightarrow[n\to\infty]{}
\int_{t}^{s}\int_{\t}^{s} \int_{\R^N} \A^{(r,\y)}\param\big(\t,e^{(\t-t)B}x;r,\y\big)
\phi(r,\y;T,y)d\y dr d\t.
\end{align}
On the other hand, by the potential estimate \eqref{eq:pot_est}, for any $n\in \N$ we have %the function%for any $\eps>0$ we have
%by the well-posedness of the potential we have that  for every $\e'>0$ there exist a positive constant $C$ such that for any $r\in\,]t,s[$ and $0<\e<r$ it holds
\begin{equation}
 \left| \int_{\R^N} \param\big(r-\e_n(r),e^{(r-\e_n(r)-t)B}x;r,\y\big)\phi(r,\y;T,y)  d\y \right|
\leq C\frac{\G^{\m+\e}(\t,e^{(\t-t)B}x;T,y)}{(T-r)^{1-\frac{\a}{2}}(T-r)^{\frac{Q}{2}}},\qquad r\in[t,s]. %{\blue \frac{1}{(T-r)^{Q/2}} }
 %\G^{{\m+\e}}\big(r-\e_n(\delta),e^{(r-\e_n(\delta)-t)B}x;T,y\big).
\end{equation}
%is bounded as a function of $$; % therefore the function $r \mapsto \int_{\R^N} \param(r-\e,e^{((r-\e)-t)B}x;r,\y)\phi(r,\y;T,y) d\y$ is integrable on $[t,s]$.
%Hence, by
Thus Lebesgue's dominated convergence theorem yields%, we have
\begin{align}
%&\lim_{n\to\infty} \int_t^{s} \(\int_{\R^N} \param(r-\e_n(r),e^{((r-\e_n(r))-t)B}x;r,\y)\phi(r,\y;T,y) d\y\) dr\\
\lim_{n\to\infty} \tilde H_n(t,x)& = \int_t^{s} \lim_{n\to\infty} \int_{\R^N}
\param\big(r-\e_n(r),e^{(r-\e_n(r)-t)B}x;r,\y\big) \phi(r,\y;T,y)  d\y\, dr
\intertext{(by %Lemma \ref{datoparam}
\eqref{condfin_param}, since $\y\mapsto\phi(r,\y;T,y)$ is a bounded and continuous function for
every $r\in[t,s]$)} & =\int_t^s \phi\big(r,e^{(r-t)B}x;T,y\big) dr.
\end{align}
We have proved that %for every $0\leq t<s\leq T$ and $x,y\in\R^N$ we have
\begin{equation}
%\begin{aligned}
H(t,x)
=%&
\int_{t}^{s}\int_{\t}^{s} \int_{\R^N} \A^{(r,\y)}\param(\t,e^{(\t-t)B}x;r,\y) \phi(r,\y;T,y)d\y dr d\t%\\
%&
+\int_t^s \phi(\t,e^{(\t-t)B}x;T,y) d\t.
%\end{aligned}
\end{equation}
This and \eqref{eq:ste103} prove the statement.
%Finally we have
%\begin{align}
%\begin{aligned}
%&\Phi(s,e^{(s-t)B}x;T,y)-\Phi(t,x;T,y)\\
%=&G(t,x)-H(t,x)\\
%=&
%-\int_t^{s} \int_{s}^{T}  \(\int_{\R^N} \A^{(r,\y)}\param(\t,e^{(\t-t)B}x;r,\y)\phi(r,\y;T,y)d\y \) dr d\t \\
%&-
%\int_{t}^{s}\int_{\t}^{s} \(\int_{\R^N} \A^{(r,\y)}\param(\t,e^{(\t-t)B}x;r,\y) \phi(r,\y;T,y)d\y\) dr + \phi(\t,e^{(\t-t)B}x;T,y) d\t,
%\end{aligned}
%\end{align}
%which proves the statement.
\end{proof}

We are now in the position to prove Theorem \ref{main}, namely that $\sol=\sol(t,x;T,y)$ defined by \eqref{FS} is a fundamental solution of $\Ac + Y$ in the sense of definition Definition \ref{fund}, and that the Gaussian bounds from \eqref{eq:gaussian_1} to \eqref{eq:gaussian_4} are satisfied.

\begin{proof}[Proof of Theorem \ref{main}% (Existence of the fundamental solution)
] Let $\sol=\sol(t,x;T,y)$ be defined by \eqref{FS}. %satisfies Definition \ref{fund}, and that .

%%%%%%%%%%%%%%%%%%%%%%%%%%%%%%%%%%%%%%%%%%%%%%%%%%
\emph{Step 1.} We show that $\sol=\sol(t,x;T,y)$ satisfies point i) of Definition \ref{fund},
namely that
%Let us show that
$\sol(\cdot,\cdot;T,y)$ is an integral solution to \eqref{eq:Yu_eq_A} on $\mathcal{S}_T$ in the sense of Definition \ref{solint}. %of the equation $Y u =-\A u$.
%First of all, by Proposition \ref{lem:deriv_theta} and the Gaussian estimates for $\mathbf{P}$, we
%have
%\red{\begin{equation}
% \sol(\cdot,\cdot;T,y), \partial_{x_i} \sol(\cdot,\cdot;T,y),\partial_{x_i x_j} \sol(\cdot,\cdot;T,y) \in L^{\infty}_{\text{\rm loc}}(\mathcal{S}_T), \qquad i,j=1,\dots,d,
%\end{equation}
%}
%for any $(T,y)\in\mathcal{S}_{\TT}$.
%Now, b
By Lemma \ref{leib}, we have
\begin{align}
\sol\big(s,e^{(s-t)B}x;T,y\big)-\sol(t,x;T,y)
& = \param(s,e^{(s-t)B}x;T,y)-\param(t,x;T,y)+\Phi(s,e^{(s-t)B}x;T,y)-\Phi(t,x;T,y)\\
& =- \int_t^{s} \Big( \A^{(r,\y)}\param(\t,e^{(\t-t)B}x;r,\y) + F(\t,x;T,y) \Big) d\t. \label{eq:ste102}
\end{align}
Furthermore, by \eqref{eq:def_F} and since $\phi(t,x;T,y)$ solves the integral equation \eqref{volt}, we obtain
\begin{align}
%\end{align}
%We remind that $\phi(t,x;T,y)$ solves the integral equation \eqref{volt}.
%Moreover, for every $(T,y)\in \mathcal{S}_{\TT}$ we have $\L \G^{(T,y)}(t,x;T,y)= \(\A-\A^{(T,y)}\) \G^{(T,y)}(t,x;T,y)$
%for almost every $(t,x)\in \mathcal{S}_T$.
%Therefore
%\begin{equation}
%\begin{aligned}
%\phi(\t,e^{(\t-t)B}x;T,y)=&
%\(\A-\A^{(T,y)}\)\param(\t,e^{(\t-t)B}x;T,y)\\
%&+\int_{\t}^{T}\int_{\R^N}\(\A-\A^{(r,\y)}\)\param(\t,e^{(\t-t)B}x;r,\y)\phi(r,\y;T,y)d\y  dr.
%\end{aligned}
%\end{equation}
%Hence,
%\begin{align*}
%&\sol(s,e^{(s-t)B}x;T,y)-\sol(t,x;T,y)\\
\A^{(r,\y)}\param(\t,e^{(\t-t)B}x;r,\y) + F(\t,x;T,y) & = \A \param(\t,e^{(\t-t)B}x;T,y)  \\
&\quad + \int_{\t}^{T}\int_{\R^N}\A \param(\t,e^{(\t-t)B}x;r,\y)\phi(r,\y;T,y)d\y  dr
\intertext{(by {Proposition \ref{lem:deriv_theta}})} %{\blue (qua in realta' serve un'osservazione tipo Remark \ref{rem:potential_etimates})})}
%& = \A \param(\t,e^{(\t-t)B}x;T,y) \\
% &\quad   + \A \(\int_{\t}^{T}\int_{\R^N}\param(\t,e^{(\t-t)B}x;r,\y)\phi(r,\y;T,y)d\y  dr\)   \\
& = \A \param(\t,e^{(\t-t)B}x;T,y)+\A \Phi(\t,e^{(\t-t)B}x;T,y) \\
& = \A \sol(\t,e^{(\t-t)B}x;T,y) ,
\end{align}
which, together with \eqref{eq:ste102}, concludes the proof.

\emph{Step 2.} We show that $\sol=\sol(t,x;T,y)$ satisfies point ii) of Definition \ref{fund}. In
light of the estimate \eqref{eq:stimephi}, it is straightforward to see that
\begin{equation}
| \Phi(t,x;T,y) | \leq C (T-t)^{\frac{\a}{2}} \G^{\m+\e}(t,x;T,y), \qquad
(T,y)\in\mathcal{S}_{\TT}, \ (t,x)\in \mathcal{S}_T.
\end{equation}
Therefore, it is enough to prove that, for any fixed $(T,y)\in ]0,\TT[ \times \R^N$, we have
\begin{equation}\label{condfin_param}
\lim_{\substack{(t,x)\to(T,y)\\t<T}}
\int_{\R^N}\param(t,x;T,\y)f(\y)d\y=f(y), \qquad f\in C_b(\R^N).
\end{equation}
%We first show that
%
%\begin{lemma}\label{datoparam}
%Let $f\in C_b(\R^N)$, $0<T\leq\TT$ and
%\begin{equation}
%u(t,x)=\int_{\R^N} \param (t,x;T,\y)f(\y)d\y.
%\end{equation}
%Then
%\begin{equation}
%\lim_{(t,x)\to(T,y)\atop t<T}u(t,x)= f(y)
%\end{equation}
%for any $0<T\leq\TT$ and $y\in\R^N$.
%\end{lemma}
%\begin{proof}
Recalling the definition of the parametrix $\param$, we add and subtract to obtain
%For every $x\in\R^N$ we have that
\begin{equation}
\begin{aligned}
\int_{\R^N}\param(t,x;T,\y)f(\y)d\y&=\int_{\R^N} \G^{(T,\y)} (t,x;T,\y)f(\y)d\y\\
&=\int_{\R^N}\G^{(T,y)} (t,x;T,\y)f(\y)d\y +
\underbrace{\int_{\R^N}\( \G^{(T,\y)} (t,x;T,\y)-\G^{(T,y)} (t,x;T,\y)\) f(\y)d\y }_{=:J(t,x)}.
\end{aligned}
\end{equation}
Furthermore, by estimate \eqref{incrementiparam}, for every $\e>0$ one has %there exists a positive constant $C$ such that
\begin{equation}
|J(t,x)|\leq C\int_{\R^N}  |y-\y|_B^\a \G^{\m+\e}(t,x;T,\y)d\y. %\\
%\intertext{(by the substitution $\v=\frac{1}{\sqrt{2(\m+\e)}}\C(t,T)^{-\frac{1}{2}}\(\y-e^{(T-t)B}x\)$)}
%&\leq C\int_{\R^N}
%\left|y- {\sqrt{2(\m+\e)}}\C(t,T)^{\frac{1}{2}}\v-e^{(T-t)B}x  \right|_B^\a e^{-|\v|^2}d\v
\end{equation}
Eventually, \eqref{condfin_param} follows from classical arguments.
%We have that
%\begin{equation}
%\begin{aligned}
%\left| {\sqrt{2(\m+\e)}}\C(t,T)^{\frac{1}{2}}\v+e^{(T-t)B}x -y \right|_B^\a
%&\leq C\(\|\C(t,T)\||\v|_B+|e^{(T-t)B}x -y|_B\)^\a\\
%&\leq C\(|\v|_B^{\a}+1\)
%\end{aligned}
%\end{equation}
%uniformly with respect to $(t,x)\in[0,T]\times \mathbb{B}(y,\rho),0<\rho\leq1$, for some positive constant $C$
%that only depends on $\rho,T,y,B$ and $\e$ (here $\mathbb{B}(y,\rho):=\{\y\in\R^N; |\y-y|<\rho\}$).
%Since $\C(t,T)^{\frac{1}{2}}\xrightarrow[t\to T^-]{}0_{N\times N}$
%and $e^{(T-t)B}\xrightarrow[t\to T^-]{}I_{N}$,
%by the dominated convergence theorem we get
%\begin{equation}
%|J(t,x)|\xrightarrow[\substack{(t,x)\to(T,y)\\t<T}]{}0.
%\end{equation}
%Finally, since $\G^{(T,y)}$%(t,x;T,y)$
%is fundamental solution of $\L^{(T,y)}$, we have
%\begin{equation}
%\lim_{(t,x)\to(T,y)\atop t<T}u(t,x)= f(y)+0.
%\end{equation}
%\end{proof}

\emph{Step 3.} We show the upper Gaussian bounds
\eqref{eq:gaussian_1}-\eqref{eq:gaussian_2}-\eqref{eq:gaussian_3} for $p$ and its derivatives. The
proof of the lower Gaussian bound \eqref{eq:gaussian_4} is similar to that of Theor. 4.7 in
\cite{peslangevin} and Section 5.1.4. in \cite{MR4355925}, thus we omit it for sake of
brevity.

{The Gaussian bounds of Proposition \ref{prop:gaussian_estimates}} % \eqref{eq:upper_gaussian}-\eqref{der1param}-\eqref{der2param}
and the definition of parametrix \eqref{eq:def_parametrix} yield the estimates
\eqref{eq:gaussian_1}-\eqref{eq:gaussian_2}-\eqref{eq:gaussian_3} for $\param=\param(t,x;T,y)$.
{The estimates of Proposition \ref{lem:deriv_theta} and the fact that $p = \param + \Phi$ conclude
the proof.}
%Furthermore, by the definition of $\Phi$ in \eqref{eq:def_Phi} we have
%\begin{equation}
%\Phi(t,x;T,y) =  \int_t^T J(t,x;\t;T,y) d\t  ,
%\end{equation}
%with $J$ defined as in \eqref{eq:def_J}. The potential estimates \red{of Proposition \ref{lem:stime_pot}} upon integrating in $\tau$, yield \eqref{eq:gaussian_1}-\eqref{eq:gaussian_2}-\eqref{eq:gaussian_3} for $\Phi=\Phi(t,x;T,y)$. The fact that $p = \param + \Phi$ concludes the proof.
\end{proof}
\begin{remark}%[\red{\bf Da mettere da qualche parte?}]
\label{rem:almost_surely}
%Let $T>0$ and %$f\in C^{2,\a}( \mathcal{S}_T)$.
Any integral solution $u$ to equation \eqref{eq:integral_sol}
%$Yu=-\A u + g$
on $\mathcal{S}_T$ in the sense of Definition \ref{solint}, is Lie-differentiable along
$Y$ almost everywhere on $\mathcal{S}_T$. %In particular, the limit $Yu(t,x)$ in \eqref{eq:limit_Lie_deriv} is
%satisfied for almost every $(t,x)\in \mathcal{S}_T$.
%\begin{equation}
%  \A u\big( \t,e^{(\t-t)B}x \big)  - g\big( \t,e^{(\t-t)B}x \big)  = \lim_{\tau \to t} \frac{u\big( \tau , e^{(\tau - t) B} x \big) - u(t,x)}{\tau - t}, \qquad \text{for almost every } (t,x)\in \mathcal{S}_T.
%\end{equation}
Indeed, the set $H_T$ of $(t,x)\in \mathcal{S}_T$ such that $Yu(t,x)$ in
\eqref{eq:limit_Lie_deriv} exists finite, is measurable as the limit
\begin{equation}
 \limsup_{\tau \to t^+} \frac{u\big( \tau , e^{(\tau - t) B} x \big) - u(t,x)}{\tau - t}
\end{equation}
is a measurable function of $(t,x)$ and the same holds for $\liminf$. This is a straightforward
consequence of the continuity of $u$ along the integral curves of $Y$. The fact that $H_T$ has
null Lebesgue measure stems from Fubini's theorem, as $u$ is absolutely continuous along the
integral curves of $Y$ and the map
\begin{equation}
(\tau,y)\mapsto (\tau,e^{\tau B} y)
\end{equation}
is a diffeomorphism on $\mathcal{S}_T$.
\end{remark}

%%%%%%%%%%%%%%%%%%%%%%%%%%%%%%%%%%%%%%%
%
%       SECTION: Proof of the main result
%
%%%%%%%%%%%%%%%%%%%%%%%%%%%%%%%%%%%%%%%
%\section{Proof of Theorem \ref{main}}

%%%%%%%%%%%%%%%%%%%%%%%%%%%%%%%%%%%%%%%
%
%       SECTION: Regularity of the fundamental solution
%
%%%%%%%%%%%%%%%%%%%%%%%%%%%%%%%%%%%%%%%
\section{Regularity of the fundamental solution} \label{regolaritasoluzione}
%NOTAZIONE%
%\begin{notation}
%For any $F:\mathcal{S}_T\to\R$ and $\g:\mathcal{S}_T\to \mathcal{S}_T$ suitably smooth, we use the notation
%\begin{equation}
%(\p_i F)(\g(\t,\x)):=\p_{x_i}F(t,x)\Big|_{(t,x)=\g(\t,\x)},\qquad \t\in[0,T].
%\end{equation}
%Note that $\p_i$ is not the total derivative of a composed function,
%but we use it to denote the derivative of a function composed with another
%($\g$ can be the identity function). We use this notation also for differential operators: %with spatial derivative $\p_{x_i}$:
%\begin{equation}
%\(\[a \p_i + b \p_j\] F\)(\g(\t,\x)):=a(\p_i F)(\g(\t,\x))+ b (\p_jF)(\g(\t,\x)).
%\end{equation}
%\end{notation}
%
%For $i,j=1,\dots,d$ we set
%\begin{align}
%\Phi_i(t,x;T,y)&:=\p_{x_i}\Phi(t,x;T,y)=\int_t^T\int_{\R^N} \p_{x_i}\param(t,x;\t,\y)\phi(\t,\y;T,y)d\y d\t,\\
%\Phi_{ij}(t,x;T,y)&:=\p_{x_i x_j}\Phi(t,x;T,y)=\int_t^T\int_{\R^N} \p_{x_i x_j}\param(t,x;\t,\y)\phi(\t,w\y;T,y)d\y d\t,
%\end{align}
%for $(t,x),(T,y)\in\mathcal{S}_{\TT}, t<T$.
%
%\begin{lemma}
%There exists a constant $C$ only dependent on $N, \TT, \m$ and $B$,
%such that for any $\e>0$ and $i,j=1,\dots,d$ we have
%\begin{align}
%\left|\Phi_{i}(s,e^{(s-t)B}x;T,y)-\Phi_{i}(t,x;T,y)\right|\leq C (s-t)^{\frac{1}{2}+\frac{\a}{2}-\e},\\% \qquad s\in[t,T[,
%\left|\Phi_{ij}(s,e^{(s-t)B}x;T,y)-\Phi_{ij}(t,x;T,y)\right|\leq C (s-t)^{\frac{\a}{2}-\e},
%\end{align}
%for every $(t,x),(T,y)\in \mathcal{S}_T, t\leq s<T$.
%\end{lemma}
%
In this section we prove Theorem \ref{ta1}. Since $\sol(\cdot,\cdot;T,y)$ can be represented as in
\eqref{FS}, we need to study the regularity of $\param(\cdot,\cdot;T,y)$ and
$\Phi(\cdot,\cdot;T,y)$. While the former term can be easily dealt with by means of the Gaussian
estimates of Appendix \ref{appendix:estimates_parametrix}, the latter has to be treated more carefully. We start
with the proof of Theorem \ref{ta1}, which is based on the regularity estimates for
$\Phi(\cdot,\cdot;T,y)$ and $\param(\cdot,\cdot;T,y)$ proved in Section \ref{sec:regularity_Phi}
and Section \ref{sec:regularity_param}, respectively.
%The first one have to be treated with more care, while the second is obtained by an easy application of the previous arguments.
%In order to prove that $\sol(\cdot,\cdot;T,y)$ belongs to $ C_B^{2,\b }(\mathcal{S}_T)$
%we need to study the regularity of $\Phi(t,x;T,y)$,
%  $$\Phi(t,x;T,y):=\int_t^T\int_{\R^N}\param(t,x;\t,\y)\phi(\t,\y;T,y)d\y d\t$$

\begin{proof}[Proof of Theorem \ref{ta1}]
Let  $\b<\a$. For fixed $(T,y)\in \mathcal{S}_{\TT}$, we set
\begin{equation}
f(t,x) :=p(t,x;\T,y), \qquad (t,x) \in \SS_T.
\end{equation}
We first note that, by definition of fundamental solution, \eqref{ae4} is satisfied with $f_Y = - \Ac f$.
Furthermore, for any $t\in]0,T[$, by \eqref{eq:norm_C2alpha} and the representation \eqref{FS} we have
\begin{align}
  \| f \|_{C_{B}^{2,\beta}(\mathcal{S}_{t})} & = N_{\param,1} + N_{\param,2}   +  N_{\Phi,1}  +  N_{\Phi,2}   ,
\end{align}
where
\begin{align}
N_{\param,1} & :=       \sum_{i=1}^{d}
  \|\partial_{x_i}     \param(\cdot,\cdot;T,y)     \|_{C^{1+\beta}_{Y}(\mathcal{S}_t)}+\sum\limits_{i,j=1}^{d}  \big(   \|  \partial_{x_i x_j}   \param(\cdot,\cdot;T,y)  \|_{C^{\beta}_{Y}(\mathcal{S}_t)} + \|  \partial_{x_i  x_j}  \param(\cdot,\cdot;T,y)  \|_{C^{\beta}_{d}(\mathcal{S}_t)} \big), \\
%&\quad  + \| \Ac \param(\cdot,\cdot;T,y) \|_{L^{\infty}([0,t];C^{\beta}_{B}(\rn))}        ,   \\
N_{\param,2} & :=     \| \Ac \param(\cdot,\cdot;T,y) \|_{L^{\infty}([0,t];C^{\beta}_{B}(\rn))}  , \\
N_{\Phi,1} & :=         \sum_{i=1}^{d}
  \|\partial_{x_i}     \Phi(\cdot,\cdot;T,y)     \|_{C^{1+\beta}_{Y}(\mathcal{S}_t)}+\sum\limits_{i,j=1}^{d}  \big(   \|  \partial_{x_i x_j}   \Phi(\cdot,\cdot;T,y)  \|_{C^{\beta}_{Y}(\mathcal{S}_t)} + \|  \partial_{x_i  x_j}  \Phi(\cdot,\cdot;T,y)  \|_{C^{\beta}_{d}(\mathcal{S}_t)} \big),\\ %\\
%&\quad  + \| \Ac \Phi(\cdot,\cdot;T,y) \|_{L^{\infty}([0,t];C^{\beta}_{B}(\rn))}  .
N_{\Phi,2} & :=   \| \Ac \Phi(\cdot,\cdot;T,y) \|_{L^{\infty}([0,t];C^{\beta}_{B}(\rn))}   .
\end{align}

Now, the estimates of Lemma \ref{lem:regularity_parametrix} yield
\begin{equation}\label{eq:bound_N1P}
N_{\param,1}  \leq \frac{C}{(T-t)^{\frac{Q+{2+\b}}{2}}}.
\end{equation}

To bound $N_{\param,2}$, first fix $i,j=1,\dots, d$ and note that, by estimate \eqref{derparam},
we obtain
\begin{equation}\label{eq:bound_der_param}
\sup_{x\in\R^N} \big|  \partial_{x_i x_j}   \param(s,x;T,y) \big| \leq \frac{C}{(T-t)^{\frac{Q+2}{2}}}, \qquad s<t. %\quad i,j=1,\dots,d.
\end{equation}
Furthermore, \eqref{eq:bound_N1P} combined with Remark \ref{rem:C0alpha_intrins} yield
\begin{equation}\label{eq:est_holder_der_param}
\sup_{x,v\in\R^N} \frac{|  \partial_{x_i x_j}   \param(s,x;T,y) -   \partial_{x_i x_j}
 \param(s,v;T,y)| }{ |x - v|^{\beta}_B  } \leq \frac{C}{(T-t)^{\frac{Q+{2+\b}}{2}}},  \qquad s<t.%, \quad i,j=1,\dots,d.
\end{equation}
Thus, by \eqref{eq:bound_der_param}-\eqref{eq:est_holder_der_param} we obtain
\begin{equation}
 \| \partial_{x_i x_j}   \param(s,\cdot;T,y)   \|_{C^{\beta}_{B}(\rn)}  \leq \frac{C}{(T-t)^{\frac{Q+{2+\b}}{2}}},  \qquad s<t,%, \quad i,j=1,\dots,d,
\end{equation}
which in turn implies
\begin{equation}
 \| \partial_{x_i x_j}   \param(\cdot,\cdot;T,y)   \|_{L^{\infty}([0,t];C^{\beta}_{B}(\rn))}  \leq \frac{C}{(T-t)^{\frac{Q+{2+\b}}{2}}}.%,  \qquad s<t%, \quad i,j=1,\dots,d,
\end{equation}
This, together with Assumption \ref{ass:regularity}, prove
\begin{equation}\label{eq:bound_N2P}
N_{\param,2}  \leq \frac{C}{(T-t)^{\frac{Q+{2+\b}}{2}}}.
\end{equation}
The bound for $N_{\Phi,1}$ stems from the estimates of Proposition \ref{prop:regularity_Phi_1},
which yield
\begin{equation}
N_{\Phi,1}  \leq \frac{C}{(T-t)^{\frac{Q+{2-(\alpha-\b)}}{2}}} \leq    \frac{C}{(T-t)^{\frac{Q+{2+\b}}{2}}}  .
\end{equation}
Eventually, the bound for $N_{\Phi,2}$ follows from the same arguments used to bound
$N_{\param,2}$.
\end{proof}

The rest of this section is devoted to the results utilized in the proof of Theorem \ref{ta1}. It is useful to introduce the following
\begin{notation}
Let $f=f(t,x;T,y)$ be a function defined for $(T,y)\in\mathcal{S}_{\TT}$ and $(t,x)\in\SS_T$, suitably differentiable w.r.t. $x$. For any $i%,j,k
=1,\dots, N$, we set
\begin{equation}
\partial_{i} f(t,x;T,y):= \partial_{x_i}  f(t,x;T,y), %\quad \partial_{ij} f(t,x;T,y):= \partial_{x_i x_j}  f(t,x;T,y), \quad \partial_{ijk} f(t,x;T,y):= \partial_{x_i x_j x_k}  f(t,x;T,y).\hspace{-15pt}
\end{equation}
and we adopt analogous notations for the higher-order derivatives.
\end{notation}
This notation is useful in order to compose partial derivatives with other functions. For instance, if $g=g(t,x)$ is a given function, then
\begin{equation}
\partial_{i} f\big(t,g(t,x);T,y\big) =  \partial_{z_i}  f(t,z;T,y)|_{z=g(t,x)}.
\end{equation}
\subsection{Regularity estimates of $\Phi$}\label{sec:regularity_Phi}
Now prove the H\"older estimates for $\Phi(\cdot,\cdot;T,y)$. We recall that $Q$ denotes the
homogeneous dimension of $\rn$ as in \eqref{ae11}.
% along the integral curves of $Y$.
%that $\Phi(\cdot,\cdot;T,y)$ is $\({(1+\b)}/{2}\)$-H\"older continuous along the integral curves of $Y$, for any $\b<\a$: %$\({\frac{1}{2}+\frac{\b }{2}}\)$-H\"older
\begin{proposition}\label{prop:regularity_Phi_1}
For every $\e>0$ and $0<\b <\a$ there exists a positive constant $C$, only dependent on $\TT,\m, B,  \e, \a, \b $ and the $\a$-H\"older norms of the coefficients,
such that, for any $i,j,k=1,\dots,d$, we have
\begin{align}
\big|\p_i\Phi(s,e^{(s-t)B}x;T,y)-\p_i\Phi(t,x;T,y)\big| %\\ \leq&
%\leq C(s-t)^{\frac{1}{2}+\frac{\b }{2}}\frac{1}{(T-s)^{1-\frac{\a-\b }{2}}} \(\frac{T-t}{T-s}\)^{{Q}/{2}}\G^{\m+\e}(t,x;T,y),\\% \qquad s\in[t,T[,
&\leq C(s-t)^{\frac{1+\beta}{2}} \frac{(T-t)^{{Q}/{2}}}{(T-s)^{\frac{Q+2-(\a-\b) }{2}}}\G^{\m+\e}(t,x;T,y),\label{dercurveinteg}\\% \qquad s\in[t,T[,
\big|\p_{ij}\Phi(s,e^{(s-t)B}x;T,y)-\p_{ij}\Phi(t,x;T,y)\big| %\\ \leq&
%\leq C (s-t)^{\frac{\b }{2}}\frac{1}{(T-s)^{1-\frac{\a-\b }{2}}} \(\frac{T-t}{T-s}\)^{{Q}/{2}} \G^{\m+\e}(t,x;T,y),
&\leq C (s-t)^{\frac{\b }{2}}\frac{(T-t)^{{Q}/{2}}}{(T-s)^{\frac{Q+2-(\a-\b) }{2}}}\G^{\m+\e}(t,x;T,y),\label{dercurveinteg2}\\% \qquad s\in[t,T[,
\left|\p_{ij}\Phi(t,x  +h \mathbf{e}_k ;T,y)-\p_{ij}\Phi(t,x;T,y)\right|
&\leq C |h|^{\b} {\frac{\G^{\m+\e}(t,x +h \mathbf{e}_k ;T,y)+\G^{\m+\e}(t,x;T,y)}{(T-t)^{\frac{2- (\a-\b) }{2}}} }\label{dercampidritti},
\end{align}
for every $(T,y)\in\mathcal{S}_{\TT}$, $(t,x)\in\mathcal{S}_{T}$, $t< s<T$ and $h\in\R$.
%$(t,x),(T,y)\in \mathcal{S}_T, t\leq s<T$.
\end{proposition}
The proof of estimates \eqref{dercurveinteg}-\eqref{dercurveinteg2} relies on the following
\begin{lemma}\label{lem:deriv_param_sol}
Let $(T,y)\in\mathcal{S}_{\TT}$. Then, for any $i=1,\dots, d$, the function $u:=\p_i \param
(\cdot,\cdot;T,y)%= \p_i \param (t,x;T,y)
$ is a strong Lie solution to the equation
\begin{equation}
\A u+ Yu = - \sum_{j=1}^{d+d_1}
b_{ji} \p_j\param (\cdot,\cdot;T,y)  \     \text{ on }\mathcal{S}_{T},
\end{equation}
in the sense of Definition \ref{solint}.
\end{lemma}
\begin{proof}
We note that
\begin{equation}
[\p_i,Y]\param(t,x;T,y) = [\p_i, \langle B x, \nabla \rangle+ \partial_t ] \param(t,x;T,y) %=\sum_{j=1}^N b_{ji}\p_j  %=\<B^{(i)},\nabla_x \>
=\sum_{j=1}^{d+d_1} b_{ji}\p_j \param(t,x;T,y),
\end{equation}
for every $x\in\R^N$ and for almost every $t\in[0,T[$, where, in the last equality, we used that $b_{ji}=0$ if $j>d+d_1$. %for any $i=1,\dots,d$.\\
While it is obvious that the previous identity holds for smooth functions of $(t,x)$, one can directly check that $\p_i\p_t \param(t,x;T,y)=\p_t\p_i \param(t,x;T,y)$ and thus the identity holds for the parametrix too.
%\red{We note that this identity holds when the commutator acts on sufficiently regular functions.
%Furthermore, by a direct computation, we have that
%\begin{equation}
%[\p_i,Y] \param(t,x;T,y)=\sum_{j=1}^{d+d_1} b_{ji}\p_j \param(t,x;T,y),
%\end{equation}
%for every $x\in\R^N$ and for almost every $t\in[0,T[$, for every fixed $(T,y)\in\mathcal{S}_{\TT}$.
%Indeed, the parametrix function is smooth with respect to the spatial variables, there exists the time derivative almost everywhere
%and one can directly check that $\p_i\p_t \param(t,x;T,y)=\p_t\p_i \param(t,x;T,y)$, for $i=1,\dots,N$.
%}
Therefore, we obtain %it holds that
\begin{align*}
\p_i\param(s,e^{(s-t)B}x;\t,\y)-\p_i\param(t,x;\t,\y)
%I(t,x;s)
%&\int_t^s F_i(r,e^{(r-t)B}x;\t,\y)Y\param(r,e^{(r-t)B}x;\t,\y)dr \\
%&+\int_t^s YF_i(r,e^{(r-t)B}x;\t,\y) \param(r,e^{(r-t)B}x;\t,\y)dr\\
=&\int_t^s (Y\p_i\param)(r,e^{(r-t)B}x;\t,\y)dr \\
=&\int_t^s \Big( ( \p_iY\param)(r,e^{(r-t)B}x;\t,\y)-[\p_i,Y]\param(r,e^{(r-t)B}x;\t,\y) \Big)dr
%\intertext{(since $\param(\cdot,\cdot;\t,\y)$ is solution in the sense of Definition \ref{solint} of $Yu=-\A^{(\t,\y)}u$ on $S_\t$)}
\intertext{(by Remark \ref{rem:eq_param})}
=& - \int_t^s \Big( ( \p_i\A^{(\t,\y)}\param)(r,e^{(r-t)B}x;\t,\y)+\sum_{j=1}^{d+d_1}
b_{ji}\p_j\param(r,e^{(r-t)B}x;\t,\y) \Big) dr
 \intertext{(since
$\p_i\A^{(\t,\y)}=\A^{(\t,\y)}\p_i$)} =& - \int_t^s \Big(
\big(\A^{(\t,\y)}\p_i\param\big)(r,e^{(r-t)B}x;\t,\y)+\sum_{j=1}^{d+d_1}
b_{ji}\p_j\param(r,e^{(r-t)B}x;\t,\y) \Big) dr.
\end{align*}
\end{proof}
{We are now in the position to prove Proposition \ref{prop:regularity_Phi_1}.} %prima era dentro alla proof
\begin{proof}[Proof of Proposition \ref{prop:regularity_Phi_1}]
Let $(T,y)\in\mathcal{S}_{\TT}$, $(t,x)\in\mathcal{S}_{T}$, $t< s<T$ and $h\in\R$ be fixed. Also fix $i,j,k\in\{1,\dots, d\}$. %and set $\tilde{\e}:=\a-\b $.
First we prove \eqref{dercurveinteg}. %We fix $\tilde{\e}:=\a-\b $. For any $(t,x),(T,y)\in \mathcal{S}_T, t\leq s<T$
By adding and subtracting, we have
\begin{equation}%\label{cavaemetti}
\begin{aligned}
&\p_i\Phi(s,e^{(s-t)B}x;T,y)-\p_i\Phi(t,x;T,y)\\
=&\int_{s}^T\int_{\R^N} \underbrace{ \Big(\p_i\param(s,e^{(s-t)B}x;\t,\y)-\p_i\param(t,x;\t,\y)\Big)  }_{=: I(\tau,\eta)} \phi(\t,\y;T,y)d\y d\t
-\underbrace{\int_{t}^{s}\int_{\R^N} \p_{i}\param(t,x;\t,\y) \phi(\t,\y;T,y)d\y d\t}_{=: L}.
\end{aligned}
\end{equation}
We consider the first term. %Denoting by $B^{(i)}$ is the $i$-th column of $B$,
%We have that
%\begin{equation}
%\left|\A^{(\t,\y)}\p_i\param(r,e^{(r-t)B}x;\t,\y)\right|
%\leq \frac{C}{(\t-r)^\frac{3}{2}}\G^{\m+\e}(r,e^{(r-t)B}x;\t,\y)
%\end{equation}
%and
%\begin{equation}
%\left|\p_j\param(r,e^{(r-t)B}x;\t,\y)\right|
%\leq \frac{C}{(\t-r)^\frac{3}{2}}\G^{\m+\e}(r,e^{(r-t)B}x;\t,\y).
%\end{equation}
%Since $
By Lemma \ref{lem:deriv_param_sol} and swapping the integrals as in the proof of Proposition \ref{leib}, we have
\begin{align*}
&\int_{s}^T\int_{\R^N} { I(\tau,\eta)} \phi(\t,\y;T,y)d\y d\t\\ =& - \int_{s}^T \int_t^s
\int_{\R^N} \bigg(  \big( \A^{(\t,\y)} \p_i\param\big)(r,e^{(r-t)B}x;\t,\y) + \sum_{j=1}^{d+d_1} b_{ji}\p_j\param(r,e^{(r-t)B}x;\t,\y) \bigg)
\phi(\t,\y;T,y)d\y dr d\t.
\end{align*}
Therefore, the estimates {of Proposition \ref{lem:stime_pot}} %\eqref{der1pot}-\eqref{der2pot}-\eqref{der3pot} in Remark \ref{rem:potential_etimates}
with $\delta= (\a-\b) /2$ yield
%results in Section \ref{appendix:potentialestimates},
%we obtain %have that %for every $\e>0$ there exists a positive constant C such that,
\begin{align}
%&
\bigg|\int_{s}^T\int_{\R^N} { I(\tau,\eta)} \phi(\t,\y;T,y)d\y d\t\bigg| %\\
&
\leq \int_{s}^T \int_t^s \frac{C}{(T-\t)^{1-\frac{\a-\b}{4}}({\t-r})^{\frac{3}{2}-\frac{\a+\b}{4}}} \G^{\m+\e}(r,e^{(r-t)B}x;T,y)dr d\t
\intertext{(by a standard estimate on $\G^{\m+\e}(r,e^{(r-t)B}x;T,y)$)}
&\leq C\underbrace{\int_{s}^T \int_t^s \frac{1}{(T-\t)^{1-\frac{\a-\b}{4}}({\t-r})^{\frac{3}{2}-\frac{\a+\b}{4}}} dr d\t}
_{=:K%(s,t,T)
} \(\frac{T-t}{T-s}\)^{{Q}/{2}} \G^{\m+\e}(t,x;T,y). \label{eq:estimate_ste1}
\end{align}
We now bound $K$:
%and by a direct computation
\begin{align}
%&\int_{s}^T \int_t^s \frac{1}{(T-\t)^{1-\frac{\tilde \eps}{4}}({\t-r})^{\frac{3}{2}-\frac{\a-{\tilde{\e}}/{2}}{2}}} dr d\t=
K%(t,s,T)
& =
\int_t^s \int_{s}^T \frac{1}{(T-\t)^{1-\frac{\a-\b}{4}}({\t-r})^{\frac{3}{2}-\frac{\a+\b}{4}}} d\t dr
%\intertext{(for any ${\tilde{\e}}/{2}>0$ such that ${\tilde{\e}}/{2}+{\tilde{\e}}/{2}<\a$)}
 \leq
\int_t^s \int_{s}^T \frac{1}{(T-\t)^{1-\frac{\a - \b}{4}}(\t-r)^{1-\frac{\a - \b}{4}}} d\t \frac{1}{({s-r})^{\frac{1}{2}-\frac{\beta}{2}}} dr
%\\
%& \leq
%\int_t^s \int_{r}^T \frac{1}{(T-\t)^{1-\frac{\tilde \eps}{4}}(\t-r)^{1-\frac{\tilde \eps}{4}}} d\t \frac{1}{({s-r})^{\frac{1}{2}-\frac{\b }{2}}}dr\\
\intertext{(solving the integral in $d \tau$)}
& \leq
C%_{\tilde{\e}}
\int_t^s \frac{1}{(T-r)^{1-\frac{\alpha-\beta}{2}}} \frac{1}{({s-r})^{\frac{1}{2}-\frac{\b }{2}} }dr  \leq
\frac{C}{(T-s)^{1-\frac{\alpha-\beta}{2}}} \int_t^s \frac{1}{({s-r})^{\frac{1 - \beta}{2}} }dr  \leq
\frac{C}{(T-s)^{1-\frac{\alpha-\beta}{2}}} (s-t)^{\frac{1+\beta}{2}}. \\ \label{eq:estimate_ste2}
\end{align}
%where
%\begin{align*}
%\kappa = \frac{\G_E\({\tilde \eps}/{4}\)\G_E\({\tilde \eps}/{4}\)}{\G_E\({\tilde{\e}}/{2}\)}=\frac{2}{\tilde{\e}} + O(1)\quad\text{as }\tilde{\e}\to0 .
%\end{align*}
%(see the Laurent expansion of the Gamma function).

On the other hand, estimate {\eqref{derpot}} with $\delta= \a - \b $ yields %that for every $\e>0$ and $\d<\a$ there exists a positive constant $C$ such that
\begin{align*}
|L|&\leq\int_t^s \frac{C}{(T-\t)^{1-\frac{\a - \b}{2}}(\t-t)^{\frac{1}{2}-\frac{\b}{2}}} d\t \, \G^{\m+\e}(t,x;T,y) \leq\frac{C}{(T-s)^{1-\frac{\a - \b}{2}}}\int_t^s \frac{1}{(\t-t)^{\frac{1-\beta}{2}}} d\t\,  \G^{\m+\e}(t,x;T,y)\\
&\leq\frac{C}{(T-s)^{1-\frac{\a - \b}{2}}}(s-t)^{\frac{1+\beta}{2}} \G^{\m+\e}(t,x;T,y).
\end{align*}
This, together with \eqref{eq:estimate_ste1}-\eqref{eq:estimate_ste2}, proves \eqref{dercurveinteg}.
Estimate \eqref{dercurveinteg2} can be obtained following the same arguments. %, noticing that $[\p_i,Y]\p_j=\p_j[\p_i,Y]$.

We finally prove \eqref{dercampidritti}.
%\end{proof}
%Now prove the H\"older estimates for $\Phi(t,x;T,y)$ with respect to $x_1,\dots,x_d$.
%\begin{proposition}\label{prop:regularity_Phi_2}
%For every $\e>0$ and $0<\b <\a$ there exists a positive constant $C$, only dependent on $\TT,\m, B, \e, \a, \b $ and the $\a$-H\"older norms of the coefficients,
%such that for any $i,j,k=1,\dots,d$ we have
%\begin{equation}
%\left|\p_{ij}\Phi(t,x  + h \mathbf{e}_k ;T,y)-\p_{ij}\Phi(t,x;T,y)\right|
%\leq C |h|^{\frac{\b }{2}} {\frac{\G^{\m+\e}(t,x + h \mathbf{e}_k ;T,y)+\G^{\m+\e}(t,x;T,y)}{(T-t)^{1-\frac{\a-\b }{2}}} }\label{dercampidritti},
%\end{equation}
%for any $0\leq t< T\leq T_0$, $x,y \in\R^N$ and $\delta\in\R$.
%\end{proposition}
%\begin{proof}
By {Proposition \ref{lem:deriv_theta}} we have%{\blue (serve un'osserv. tipo Remark \ref{rem:potential_etimates})}
\begin{align}
\p_{ij}\Phi(t,x +h \mathbf{e}_k ;T,y) - \p_{ij}\Phi (t,x;T,y)=\int_t^T
\underbrace{\int_{\R^N}\big( \p_{ij}\param (t,x +h \mathbf{e}_k ;\t,\y) - \p_{ij}\param (t,x;\t,\y)\big)\phi(\t,\y;T,y)d\y}_{=:I(\t)} d\t.
\end{align}
%We fix $\e>0,0<\tilde{\e}/2<\a$ and $0<\tilde{\e}/2<\a-\tilde{\e}/2$.
%We fix $\tilde{\e}:=\a-\b$.
We first prove that %there exists a positive constant $C$, that only depends on $N, \TT, \m, B,\tilde{\e}/2$ and $\e$, such that
\begin{equation}%\b==\a-\tilde{\e}/4
|I(\t)|\leq C \frac{|h|^{\b}}{(T-\t)^{1-\frac{\a - \b}{4}}({\t-t})^{1-\frac{\a - \b}{4}}}\(\G^{\m+\e}(t, x + h \mathbf{e}_k  ;T,y)+\G^{\m+\e}(t,{x};T,y)\), \qquad \tau\in ]t,T[.
\end{equation}
We consider the case $\t-t\geq h^2$. By the mean-value theorem, there exists a real $\bar h$ with
$|\bar h|\leq |h|$ such that
%Since $x-v\in\text{span}\{e_k\}$, there exists $\bar{x}$ in the segment of endpoints $x$ and $v$ such that
\begin{equation}
\left| \p_{ij}\param (t,x + h \mathbf{e}_k ;\t,\y) - \p_{ij}\param(t,x;\t,\y)\right|=|h| \left| \p_{ijk}\param (t,x +\bar h \mathbf{e}_k ;\t,\y)\right|.
\end{equation}
Therefore, by the estimate {\eqref{derpot}} with $\delta = (\a-\b) /2$, we have %for any $0<\tilde{\e}/2<\a-\tilde{\e}/2$,
\begin{align}
|I(\t)|& \leq C \frac{|h|}{(T-\t)^{1-\frac{\a - \b}{4}}({\t-t})^{\frac{3}{2}-\frac{\a +
\b}{4}}}\G^{\m+\e}(t,x +\bar h \mathbf{e}_k ;T,y)
 \intertext{(since $\t-t\geq h^2$)}
  & \leq C \frac{|h|^{\b}}{(T-\t)^{1-\frac{\a - \b}{4}}({\t-t})^{1-\frac{\a - \b}{4}}}\G^{\m+\e}(t,x +\bar h \mathbf{e}_k ;T,y)
 \intertext{(by standard estimates on $\G^{\m+\e}(t,x +\bar h \mathbf{e}_k ;T,y)$ with $\t-t\geq h^2$)}
  & \leq C \frac{|h|^{\b}}{(T-\t)^{1-\frac{\a - \b}{4}}({\t-t})^{1-\frac{\a - \b}{4}}}\big(\G^{\m+\e}(t, x+h \mathbf{e}_k;T,y)+\G^{\m+\e}(t,x;T,y)\big).%, \qquad \tau\in]t,T[.
\end{align}
We now consider the case $\t-t< h^2$. Employing triangular inequality and estimate
{\eqref{derpot}} with $\delta = (\a - \b) /2$, we get
\begin{align*}
|I(\t)|&\leq \frac{C}{(T-\t)^{1-\frac{\a-\b}{4}}({\t-t})^{1-\frac{\a+\b}{4}}}\big(\G^{\m+\e}(t, x+
h \mathbf{e}_k;T,y)+\G^{\m+\e}(t,x;T,y)\big) \intertext{(since $\t-t<h^2$)} & \leq C
\frac{|h|^{\b}}{(T-\t)^{1-\frac{\a-\b}{4}}({\t-t})^{1-\frac{\a-\b}{4}}} \big(\G^{\m+\e}(t, x+ h
\mathbf{e}_k;T,y)+\G^{\m+\e}(t,x;T,y)\big).%, \qquad \tau\in]t,T[.
\end{align*}
Therefore, combining the previous estimates, we obtain %both cases $\t-t< h^2$ and $\t-t\geq  h^2$, we obtain
\begin{align*}
\bigg |\int_t^T I(\t)d\t\bigg|&\leq C{|h|^{\b}}
\int_t^T\frac{1}{(T-\t)^{1-\frac{\a-\b}{4}}({\t-t})^{1-\frac{\a-\b}{4}}}d\t \big(\G^{\m+\e}(t, x+
h \mathbf{e}_k;T,y)+\G^{\m+\e}(t,x;T,y)\big)   \\ &\leq C{|h|^{\b}}
\frac{1}{(T-t)^{1-\frac{\a-\b}{2}}} \big(\G^{\m+\e}(t, x+ h
\mathbf{e}_k;T,y)+\G^{\m+\e}(t,x;T,y)\big),
\end{align*}
which proves \eqref{dercampidritti}. %is proved.
\end{proof}

\subsection{Regularity estimates for the parametrix}\label{sec:regularity_param}
%On the other hand, f
We have the following H\"older estimates for $\param$. %the following results can be proved with similar arguments as above:
\begin{lemma}\label{lem:regularity_parametrix}
Let $0\leq\b\leq\a$.
Then for every $\e>0$ there exists a positive constant $C$, only dependent on $\TT,\m, B,  \e, \a , \b$ and the $\a$-H\"older norms of the coefficients,
such that for any $i,j,k=1,\dots,d$ we have
\begin{align}
\big|\p_i\param(s,e^{(s-t)B}x;T,y)- \p_i\param (t,x;T,y)\big| %\\ \leq&
%&\leq C(s-t)^{\frac{1}{2}+\frac{\b}{2}}\frac{1}{(T-s)^{1+\frac{\b}{2}}} \(\frac{T-t}{T-s}\)^{{Q}/{2}}\G^{\m+\e}(t,x;T,y), \label{hold1param}\\% \qquad s\in[t,T[,
&\leq C(s-t)^{\frac{1+\beta}{2}}\frac{(T-t)^{{Q}/{2}}}{(T-s)^{\frac{Q+2+\b }{2}}}\G^{\m+\e}(t,x;T,y), \label{hold1param}\\% \qquad s\in[t,T[,
\big|\p_{ij}\param(s,e^{(s-t)B}x;T,y)- \p_{ij}\param (t,x;T,y)\big| %\\ \leq&
%&\leq C (s-t)^{\frac{\b}{2}}\frac{1}{(T-s)^{1+\frac{\b}{2}}} \(\frac{T-t}{T-s}\)^{{Q}/{2}} \G^{\m+\e}(t,x;T,y),\label{hold2param}\\
&\leq C (s-t)^{\frac{\b}{2}}\frac{(T-t)^{{Q}/{2}}}{(T-s)^{\frac{Q+2+\b }{2}}}\G^{\m+\e}(t,x;T,y),\label{hold2param}\\
%\end{align}
%and
%\begin{align}
\big|\p_{ij}\param(t, x+ h \mathbf{e}_k;T,y)- \p_{ij}\param(t,x;T,y)\big| &\leq C
|h|^{\b} {\frac{1}{(T-t)^{\frac{2+\b}{2}}} \(\G^{\m+\e}(t, x+ h
\mathbf{e}_k;T,y)+\G^{\m+\e}(t,x;T,y)\)} ,\\
& \label{hold3param} % \(\frac{T-t}{T-s}\)^{{Q}/{2}}
\end{align}
for any $(T,y)\in\mathcal{S}_{\TT}$, $(t,x)\in\mathcal{S}_{T}$, $t< s<T$
% $0\leq t<s< T$, $x,y \in\R^N$
and $h\in\R$.
%$v=x+h e_k, h\in\R$.
%x-y\in\text{span}\{e_1,\dots,e_{d}\}$
\end{lemma}
%for every $(t,x),(T,y)\in \mathcal{S}_T, t\leq s<T$.
\begin{proof}
%We note that $[\p_i,Y]=\sum_{j=1}^N b_{ji}\p_j=\<B^{(i)},\nabla_x
%\>=\sum_{j=1}^{d+d_1} b_{ji}\p_j$, since $b_{ij}=0$ if $j>d+d_1$ for any $i=1,\dots,d$ (where
%$B^{(i)}$ is the $i$-th column of $B$). Then it holds that
%Let $0\leq t < T\leq T_0$, $x,y \in\R^N$, $\d\in\R$ and $i,j,k \in \{  1,\dots, d  \}$ be fixed
%throughout the proof.
We first consider \eqref{hold1param}.
%Proceeding as in the previous proof
By Lemma \ref{lem:deriv_param_sol} we have
\begin{align*}
\p_i\param(s,e^{(s-t)B}x;T,y)-\p_i\param(t,x;T,y)
%=&\int_t^s (Y\p_i\param)(r,e^{(r-t)B}x;T,y)dr \\
%=&\int_t^s( \p_iY\param)(r,e^{(r-t)B}x;T,y)-[\p_i,Y]\param(r,e^{(r-t)B}x;T,y)dr \\
%\intertext{(since $\param(\cdot,\cdot;T,y)$ is solution in the sense of Definition \ref{solint} of $Yu=-\A^{(T,y)}u$ on $S_\t$)}
%=&\int_t^s -( \p_i\A^{(T,y)}\param)(r,e^{(r-t)B}x;T,y)-\sum_{j=1}^{d+d_1}
%b_{ji}\(\p_j\param\)(r,e^{(r-t)B}x;T,y)dr \\
%\intertext{(since $\p_i\A^{(T,y)}=\A^{(T,y)}\p_i$)}
= - \int_t^s \Big( \A^{(T,y)}\p_i\param(r,e^{(r-t)B}x;T,y) + \sum_{j=1}^{d+d_1}
b_{ji}\p_j\param(r,e^{(r-t)B}x;T,y)\Big)dr.
\end{align*}
Therefore, by boundedness of the coefficients of $\A^{(T,y)}$ and the estimates %\eqref{der1param}-\eqref{der2param}-\eqref{der3param}-\eqref{der33param}
{of Proposition \ref{prop:gaussian_estimates}}, we obtain
\begin{align*}
  \big|\p_i\param(s,e^{(s-t)B}x;T,y)-\p_i\param(t,x;T,y)\big| &\leq \int_t^s
  \frac{C}{(T-r)^\frac{3}{2}}\G^{\m+\e}(r,e^{(r-t)B}x;T,y)dr\\
   &\leq \int_t^s \frac{C}{(T-r)^\frac{3}{2}}dr\(\frac{T-t}{T-s}\)^{{Q}/{2}}\G^{\m+\e}(t,x;T,y)
\intertext{(for any $\b\le 1$)}% &\leq C \frac{(s-t)}{(T-s)^{\frac{3}{2}}}\(\frac{T-t}{T-s}\)^{{Q}/{2}} \G^{\m+\e}(t,x;T,y)\\
&\leq C \frac{(s-t)^{\frac{1+\beta}{2}}}{(T-s)^{1+\frac{\b}{2}}}\(\frac{T-t}{T-s}\)^{{Q}/{2}}
\G^{\m+\e}(t,x;T,y).
\end{align*}
The proof of \eqref{hold2param} is based on analogous arguments.%analogous.

We finally prove \eqref{hold3param}. %We fix $x,v\in\R^N$ and we
As for \eqref{dercampidritti}, we first consider the case $T-t\geq h^2$. By the mean-value
theorem, there exists a real $\bar{h}$ with $|\bar{h}|\leq |h|$ such that
%Since $x-v\in\text{span}\{e_k\}$, there exists $\bar{x}$ in the segment of endpoints $x$ and $v$ such that
\begin{align}
\left| \p_{ij}\param (t, x+ h \mathbf{e}_k; T, y)- \p_{ij}\param (t,x; T, y)\right| &=
|h|\left|\p_{ijk}\param(t, x+\bar h \mathbf{e}_k ; T, y)\right| \intertext{(by estimate
\eqref{derparam})} & \leq C \frac{|h|}{({T-t})^{\frac{3}{2}}}\G^{\m+\e}(t, x+\bar h \mathbf{e}_k
;T,y) \intertext{(since $ T-t\geq h^2$ and by standard estimates on $\G^{\m+\e}(t, x+\bar h
\mathbf{e}_k ;T,y)$)} &\leq C \frac{|h|^{\b}}{({ T-t})^{1+\frac{\b}{2}}}\(\G^{\m+\e}(t, x+ h
\mathbf{e}_k;T,y)+\G^{\m+\e}(t,x;T,y)\).
\end{align}
We now consider $ T-t< h^2$. Employing triangular inequality and estimate \eqref{derparam} yields %for $0<\e'<\b-\d$
%we have that
\begin{align*}
\left| \p_{ij}\param (t,  x+ h \mathbf{e}_k ; T, y)- \p_{ij}\param (t,x; T, y)\right| &\leq
\frac{C}{{T-t}}\(\G^{\m+\e}(t, x+ h \mathbf{e}_k;T,y)+\G^{\m+\e}(t,x;T,y)\) \intertext{(since $
T-t< h^2$)} &\leq C \frac{|h|^{\b}}{({ T-t})^{1+\frac{\b}{2}}}\(\G^{\m+\e}(t, x+ h
\mathbf{e}_k;T,y)+\G^{\m+\e}(t,x;T,y)\).
\end{align*}
%By combining both cases $ T-t< h^2$ and $ T-t\geq  h^2$ we obtain
This concludes the proof of \eqref{hold3param}.
\end{proof}

\section{Cauchy problem}\label{sec:Cau}
We consider the Cauchy problem
\begin{equation}\label{eq:Cauchy_prob}
 \begin{cases}
  \A u+Yu = f &\text{on }\mathcal{S}_T,\\
  u(T,\cdot)= g &\text{on }\R^N.
\end{cases}
\end{equation}
In this section we collect a few results for \eqref{eq:Cauchy_prob} that follow from our main
Theorems \ref{main} and \ref{ta1}. For sake of brevity, we omit the proofs that are based on
rather standard arguments.
\begin{assumption}\label{growholdloc}
%Let $T\in\,]0,\TT]$.
$f$ is measurable function on $\mathcal{S}_T$, $g\in C(\rn)$ %$f\in L^{\infty}_{\text{\rm loc}}(\mathcal{S}_T)$ %(for some fixed $T\in\,]0,\TT]$)
and there exists $\b>0$ such that, for a.e. $t\in[0,\TT]$:
\begin{itemize}
\item[a)] there exists a positive constant $C$ %, \red{with $C_{2}$ suitably small}, %such that $C_2<\frac{1}{4\m T}$,
such that %\red{\bf[Norma Euclidea o omogenea?]}
\begin{align}\label{ae10}
 |f(t,x)|+|g(x)|&\leq C e^{C |x|^2},\qquad x\in\rn;
\end{align}
%for any $x\in\rn$;
%for every $(t,x)\in\,]0,T[\times\R^N$, where $C_1$ and $C_2$ are positive constants, with
%$C_2<\frac{1}{4\m T}$;
\item[b)] for every compact subset $K$ of $\rn$, there exists $C_{K}>0$ such that
\begin{equation}
 |f(t,x)-f(t,y)|\leq C_{K}|x-y|^\b_B, \qquad %t \in[0,T[,\
 x,y\in K.
\end{equation}
%for every $(t,x)\in\,]0,T[\times\R^N$, $y\in\R^N$.
%where $B_n(0)$ is the open ball of $\R^N$ of radius $n$ centered at the origin.
\end{itemize}
\end{assumption}
%We are now ready to state our second main result.
\begin{proposition}\label{main_bis}
%Let $\L$ be as in \eqref{L} and let Assumptions \ref{coer}, \ref{blocks} and \ref{ass:regularity} be in force.
%Under Assumptions of Theorem \ref{main}, for any fixed $T \in\,]0,\TT]$, $f\in C_b(\R^N)$ and
%$g\in L^{\infty}_{\text{\rm loc}}(\mathcal{S}_T)$ satisfying Assumption \ref{growholdloc}, the
%function
%Let  %be a bounded and continuous function on $\rn$ and
%and $f$ as in Assumption .
Under Assumptions \ref{coer}, \ref{ass:hypo}, \ref{ass:regularity} and \ref{growholdloc}, there
exists $T>0$ such that the function
\begin{equation}
 u(t,x):=\int_{\R^N}\sol(t,x;T,y)g(y)dy-\int_t^T\int_{\R^N} p(t,x;s,y)f(s,y)dy ds %\qquad (t,x)\in \mathcal{S}_T. %[0,T[\times\R^N,
\end{equation}
belongs to $C_{B,\text{\rm loc}}^{2,\a}(\mathcal{S}_T)\cap C([0,T]\times\rn)$ and is the unique
solution of the Cauchy problem \eqref{eq:Cauchy_prob}, in the sense of Definition \ref{solint},
satisfying the growth estimate \eqref{ae10} for some positive constant $C$.
 %{\blue (definiamo $C_{B{\blue , \text{loc}}}^{2,\a}(\mathcal{S}_T)$?)}
\end{proposition}
%\begin{remark}
%Once more, Remark \ref{rem:almost_surely} tells us that the PDE \eqref{eq:integral_sol} is
%pointwise solved almost everywhere on $\mathcal{S}_T$.
%\end{remark}
The following result contains further useful properties that allow to view the fundamental
solution as the transition probability density of a Markovian process.
\begin{proposition}\label{main_ter}
%Let $\L$ be as in \eqref{L} and let Assumptions \ref{coer}, \ref{blocks} and \ref{ass:regularity} be in force.
Under the assumptions of Theorem \ref{main} we have:
\begin{itemize}
\item[i)] the Chapman-Kolmogorov identity
\begin{equation}
 \sol(t,x;\T,y)=\int_{\R^N}\sol(t,x;s,\y)\sol(s,\y;\T,y)d\y, \qquad t< s< \T,\ x,y\in\R^N;
\end{equation}
%for every $(t,x)\in \mathcal{S}_T, \t\in\,]t,T[$;
\item[ii)] if the zeroth order coefficient $a$ of $\A$ is constant, i.e. $a(t,x)= \bar{a}$, then
\begin{equation}
  \int_{\R^N}p(t,x;\T,y)dy=e^{\bar{a}(\T-t)},\qquad t <\T,\ x\in\R^N.
\end{equation}
%for any $\quad (t,x)\in \mathcal{S}_T$. %[0,T[\times\R^N.
\end{itemize}
\end{proposition}

\appendix
\section{Gaussian estimates}\label{appendix:estimates_parametrix}
%In this appendix we present some results that are preliminary to the parametrix technique presented in Section \ref{sec:parametrix}.
%Here we follow the ideas in \cite[Section 3]{difpas}, but with some technical subtleties.
%In the explicit expression of $\G^{(s,v)}(t,x;T,y)$ (see \eqref{eq:param}) the $x$ variable appears multiplied by the matrix $e^{(T-t)B}$.
%Therefore, when we compute the spatial derivatives of $\G^{(s,v)}(t,x;T,y)$ we do not obtain a factor given by the entries of the covariance matrix,
%as in \cite{difpas}: this is due to the fact that $\A+Y$ acts on the backward variables of $\G^{(s,v)}(t,x;T,y)$, unlike the aforementioned paper.
%%is a backward Kolmogorov operator, since the term $\p_t$ appears with a plus...
%In order to have an easier computation of the spatial derivative of the parametrix we will study a modification of $\C^{(s,v)}(t,T)$.
%We also note that the fact that the coefficients of $A^{(s,v)}$ in \eqref{aao} are only measurable in time is irrelevant at this level.

%{
We prove Gaussian estimates that are crucial in the analysis of Sections \ref{sec:parametrix} and
\ref{regolaritasoluzione}.
%\are preliminary to the parametrix technique presented in Section \ref{sec:parametrix}.
Here we follow the ideas in \cite[Section 3]{difpas}, but with some technical difference. Namely,
in the aforementioned paper the Kolmogorov operator acts on the forward variables of
$\G^{(s,v)}(t,x;T,y)$, whereas here
%This implies that when we compute the spatial derivatives of $\G^{(s,v)}(t,x;T,y)$ the argument of the exponential gives a factor that depends on the entries of covariance matrix $\C^{(s,v)}(t,T)$.
%Instead, in our work
we consider $\A+Y$ acting on the backward variables $(t,x)$. %As a consequence, in the explicit expression of $\G^{(s,v)}(t,x;T,y)$ (see \eqref{eq:param}) the $x$ variable appears multiplied by the matrix $e^{(T-t)B}$.
This has an impact on the spatial derivatives, which contain additional factors that require a
careful analysis.
%
%Therefore
%the spatial derivatives will contain a factor that depends on a transformation of $\C^{(s,v)}(t,T)$, that we will have to study.
%We also point out that at this level the fact that the coefficients of $A^{(s,v)}$ in \eqref{aao} are only measurable in time is irrelevant.} \red{(Toglierei quest'ultima frase)}

Throughout the appendix we suppose that Assumptions \ref{coer}, \ref{ass:hypo} and
\ref{ass:regularity} are satisfied and fix $(s,v)\in\mathcal{S}_{\TT}$. Denoting by $B_0$ the
matrix $B$ with null $\ast$-blocks, we define {the $N\times N$ matrices}
\begin{align}
\C_0(t):=& \int_{0}^{t} e^{(t-\t) B_0}
\begin{pmatrix} I_{d} & 0 \\0 & 0\end{pmatrix}
e^{(t-\t) B_0^*}d\t, %\qquad 0\leq t\leq\TT,
\\
%\end{equation}
%and, for $(s,v)\in\mathcal{S}_{\TT}$,
%\begin{equation}
\C_0^{(s,v)}(t,T):=& \int_{t}^{T} e^{(T-\t) B_0} A^{(s,v)} (\t) e^{(T-\t) B_0^*}d\t, %, \qquad 0\leq t\red{\leq}T\leq\TT.
\end{align}
with $A^{(s,v)}$ as defined in \eqref{aao}.
%\pur{(non mi sembra utile mettere t=T)}
{As an immediate consequence of Assumption \ref{coer} we can compare the quadratic forms associated to $\C^{(s,v)}$ (as in \eqref{cao}), $\C_0^{(s,v)}$ with $\C(T-t)$ (as in \eqref{eq:Ct}), $\C_0(T-t)$, respectively:}
\begin{align}
\frac{1}{\m}\C(T-t)\leq\C^{(s,v)}(t,T)\leq\m\C(T-t), \label{confronto1}\\
\frac{1}{\m}\C_0(T-t)\leq\C_0^{(s,v)}(t,T)\leq\m\C_0(T-t), \label{confronto0}
\end{align}
for any $ t \leq T%\leq\TT
$. %$(s,v)\in\mathcal{S}_{\TT}$,
%\pur{(non mi sembra utile mettere t=T)}
Moreover, an asymptotic comparison near 0 of $\C^{(s,v)}$ and $\C_0^{(s,v)}$ holds:
\begin{lemma}\label{asint1}
There exist two positive constants $C$ and $\d$, only dependent on $\m$ and $B$, such that %for every $t,T\in[0,\TT]$, $0<T-t< \d$, we have
\begin{align}
%\big(1-(T-t) C \big) \C_0^{(s,v)}(T-t) &\leq \C^{(s,v)}(t,T) \leq \big(1+(T-t) C\big) \C_0^{(s,v)}(T-t),\\
\frac{1}{2\m}\C_0(T-t)&\leq \C^{(s,v)}(t,T)\leq 2\m\C_0(T-t),\\
\frac{1}{\(2\m\)^{N}}\det\C_0(T-t)&\leq \det\C^{(s,v)}(t,T)\leq \(2\m\)^{N}\det\C_0(T-t),
\end{align}
for any $0<T-t< \d$. Analogous estimates hold for $\big(\C^{(s,v)}(t,T)\big)^{-1}$.
%for any $(s,v)\in\mathcal{S}_{\TT}$.
\end{lemma}
\begin{proof}
It follows from the same arguments of \cite[Lemma 3.1]{lanpol}: the proof is only based on the
properties of the matrices $A$ and $B$, and it is not relevant whether $A$ has constant or
time-dependent entries.
%and the fact that $A$ has constant coefficients is not relevant.
%Therefore we can repeat the arguments \red{repalcing $A$ with $A^{(s,v)}(\t)\pur{,\t<\d_0}$}.
%using the matrix $A^{(s,v)}(\t)$, for any $\t\in[t,T]$, instead of $A$.
\end{proof}

%\begin{corollary}\label{asint1}
%There exists a positive constant $\d$, only dependent on
%$\m$ and $B$, such that %for any $t,T\in[0,\TT]$, $0<T-t< \d$, we have
%\begin{align}
%\frac{1}{2\m}\C_0(T-t)&\leq \C^{(s,v)}(t,T)\leq 2\m\C_0(T-t),\\
%\frac{1}{\(2\m\)^{N}}\det\C_0(T-t)&\leq \det\C^{(s,v)}(t,T)\leq \(2\m\)^{N}\det\C_0(T-t),
%\end{align}
%for any $0<T-t< \d$. Analogous estimates hold for $\big(\C^{(s,v)}(t,T)\big)^{-1}$.
%\end{corollary}
%\begin{proof}
%It is a direct consequence of \eqref{confronto0} and Lemma \ref{asint}, choosing $\d$ suitably small.
%\end{proof}
%We note that $|\cdot|_B$ is homogeneous with respect to the family of dilation defined by
%\begin{equation}\label{dilations}
%D(\d):=\text{diag}(\d I_{d},\d^3 I_{d_1},\dots,\d^{2r+1}I_{d_{\rr}}),\qquad\d>0.
%\end{equation}
%This fact is very useful in order to prove the estimates for the parametrix technique.
\begin{remark}\label{remark:stimedet}
We note that $|\cdot|_B$ is homogeneous with respect to the family of dilations defined by the matrices %$D(\lambda)$.
\begin{equation}\label{dilations}
D(\lambda):=\text{diag}(\lambda I_{d},\lambda^3 I_{d_1},\dots,\lambda^{2\rr+1}I_{d_{\rr}}), \qquad \lambda\geq 0.
\end{equation}
%For every $\lambda>0$ we define {the $N\times N$ matrix} %(ma $r$ sarebbe $q$?)}
%This fact is very useful in order to prove the estimates for the parametrix technique.
In \cite[Proposition 2.3]{lanpol} it is proved that%, for every $t\geq0$,
\begin{equation}\label{eq:covmatomog}
\C_0(t)= D(\sqrt{t})\C_0(1)D(\sqrt{t}), \qquad t\geq0.
\end{equation}
Therefore, %we have that,
for $0<T-t<\d$ with $\delta$ as in Lemma \ref{asint1},
\begin{equation}
\frac{(T-t)^Q}{\(2\m\)^{N}}\det\C_0(1)\leq \det\C^{(s,v)}(t,T)\leq \(2\m\)^{N}(T-t)^Q\det\C_0(1).
\end{equation}
\end{remark}

%We now note that in the definition of $\G^{(s,v)}(t,x;T,y)=\gg(\C ^{(s,v)}(t,T),y-e^{(T-t)B}x)$
%the term $x$ appear in the exponent multiplied by the matrix $e^{(T-t)B}$.
%In order to easily
To compute the spatial derivatives of $\G^{(s,v)}(t,x;T,y)$ it is useful noticing that%=\gg(\C ^{(s,v)}(t,T),y-e^{(T-t)B}x)
\begin{equation}
\G^{(s,v)}(t,x;T,y)=\gg\big( H^{(s,v)}(t,T), e^{-(T-t)B}y-x \big),\qquad
{(T,y)\in\mathcal{S}_{\TT}, \ (t,x)\in\SS_T,}
%(t,x),(T,y)\in\mathcal{S}_{\TT},t<T,
\end{equation}
where% we define
\begin{equation}\label{H}
H^{(s,v)}(t,T):= e^{-(T-t)B}\C^{(s,v)}(t,T)e^{-(T-t)B^*}.
\end{equation}
Since $\C^{(s,v)}(t,T)$ is symmetric positive definite and $e^{-(T-t)B}$ is non-singular, then
$H^{(s,v)}(t,T)$ is symmetric and positive definite for every $0\leq t < T$. %\pur{(NO: per t=T non è vero che C o H sono definite positive e non singolari)}

In order to give estimates on the matrix $H^{(s,v)}$ we need to study the elements of $e^{tB}$.
%We recall the notation
%\begin{equation}
%{\bar{d}_j:= \sum_{k=0}^j d_k},\quad  {\bar{d}_{-1}:=0}
%\end{equation}
%\begin{lemma}
%Let $h,k=1,\dots,r$, $i\in\{\bar{d}_{h-1}+1,\dots,\bar{d}_h\}$ and $j\in\{\bar{d}_{k-1}+1,\dots,\bar{d}_k\}$.
%If $h-k>n$, then, for every $n\in\N$,
%\begin{equation}
%\(B^N\)_{ij}=0.
%\end{equation}
%\end{lemma}
We recall the block partition \eqref{B} of the matrix $B$: %structure
for $h,k=0,\dots,\rr $, we denote the $d_h\times d_k$ block of $B$ by%with
\begin{equation}
\mathcal{Q}_{hk}:=\(b_{ij}\)_{ i=\ddd_{h-1}+1,\dots,\ddd_h\atop j=\ddd_{k-1}+1,\dots,\ddd_k},
\end{equation}
%recalling the definition of
with $\ddd_{h}$ %, for $h=0,\dots,\rr,$
as in \eqref{e7}.
Note that by \eqref{B} we have
\begin{equation}\label{Qprop}
\begin{cases}
\mathcal{Q}_{hk}=0_{d_h\times d_k} &\text{if } h>k+1,\\ \mathcal{Q}_{hk}=B_h  &\text{if } h=k+1,\\
\mathcal{Q}_{hk}=\ast &\text{if } h<k+1.
\end{cases}
\end{equation}
Analogously, for $n\in\N$, we can consider the same block decomposition for $B^n$. We denote by $\mathcal{Q}_{hk}^{(n)}$ the $d_h\times d_k$ block of $B^n$.
%\begin{equation}
%\mathcal{Q}_{hk}^{(n)}:=\big(\(B^n\)_{ij}\big)_{ i=\ddd_{h-1}+1,\dots,\ddd_h\atop
%j=\ddd_{k-1}+1,\dots,\ddd_k}.
%\end{equation}

\begin{lemma}\label{potenzeblocchi}
Let $h,k=0,\dots,\rr $ and $n\in\N$. Then %If $h>k+n$, then
\begin{equation}\label{eq:Qn}
\mathcal{Q}^{(n)}_{hk}=0_{d_h\times d_k}, \qquad h>k+n,
\end{equation}
{which is
%\begin{equation}
$\(B^n\)_{ij}=0$
%\end{equation}
if $i\in\{\ddd_{h-1}+1,\dots,\ddd_h\}$ and $j\in\{\ddd_{k-1}+1,\dots,\ddd_k\}$.}
\end{lemma}
\begin{proof}We proceed by induction on $n$. The case of $n=1$ is obvious (see \eqref{Qprop}).
%We prove the thesis by induction on $n\geq2$.
%Let $h>k+2$. By the block matrix multiplication rule, we have that
%\begin{equation}
%\mathcal{Q}^{(2)}_{hk}=\sum_{s=0}^{\rr } \mathcal{Q}_{hs}\mathcal{Q}_{sk}.
%\end{equation}
%If $s< h-1$, then $h>s+1$ and $\mathcal{Q}_{hs}=0_{d_h\times d_s}$; if $s\geq h-1$, then $s>k+1$
%and $\mathcal{Q}_{sk}=0_{d_s\times d_k}$. Therefore $\mathcal{Q}^{(2)}_{hk}=0_{d_h\times d_k}$.
Now we assume that  \eqref{eq:Qn}  holds for a certain $n\in\N$.
%$ n = \bar n-1$ and we we prove that it holds for $n=\bar n$. %\pur{ and we prove the induction step}.
For $h>k+n+1$ we have %proceeding as above,
\begin{equation}
\mathcal{Q}^{(n+1)}_{hk}%=\mathcal{Q}^{(n)}_{hk}\mathcal{Q}_{hk}
= \sum_{m=0}^{\rr } \mathcal{Q}^{(n)}_{hm}\mathcal{Q}_{mk}.
\end{equation}
If $m < h- n$, then $\mathcal{Q}^{(n)}_{hm}=0_{d_h\times d_m}$ by inductive hypothesis; if
$m\geq h- n $, then $m>k+1$ and $\mathcal{Q}_{mk}=0_{d_m\times d_k}$. Therefore
$\mathcal{Q}^{(n+1)}_{hk}=0_{d_h\times d_k}$.
\end{proof}
%\begin{remark}
%Under the hypothesis of Lemma \ref{potenzeblocchi}, by the same arguments we can prove that if
%$h-k=n$, then $\mathcal{Q}^{(n)}_{hk}=B_{h}B_{h+1} \cdots  \underbrace{B_{h+{n-1}}}_{=B_{k-1}}$. \red{(?)}
%\pur{This means that the $h,h+1$ block of $B^n$ is non singular.}
%\end{remark}
\begin{lemma}\label{terminiexp}
Let $h,k=1,\dots,\rr$  such that $h-k=: n \in \N$. For any $i\in\{\bar{d}_{h-1}+1,\dots,\bar{d}_h\}$ and
$j\in\{\bar{d}_{k-1}+1,\dots,\bar{d}_k\}$ we have %\red{such that} $h-k=n$.
%Then
\begin{equation}
\(e^{tB}\)_{ij}=O(t^n),\quad\text{as }t\to0.
\end{equation}
\end{lemma}
\begin{proof}%\red{(sostituiamo $s$ con $m$?)}
%By definition of matrix exponential, we have
%\begin{equation}
%%e^{tB}=I_{N\times N}+tB+t^2B^2+\dots+t^{n-1}B^{n-1}+t^nB^n+ O(t^{n+1}).
%e^{tB}=\sum_{m=0}^{n-1}\frac{t^mB^m}{m!}+\frac{t^nB^n}{n!} + O(t^{n+1}),\quad\text{as } t\to0. %O_{N\times N}(t^{n+1})
%\end{equation}
From Lemma \ref{potenzeblocchi} we have that $(B^m)_{ij}=0$ for every $m=0,\dots,n-1$, since
$\mathcal{Q}^{(m)}_{hk}=0_{d_h \times d_k}$ for $h-k=n>m$. Therefore
\begin{equation}
(e^{tB})_{ij}=\frac{t^n(B^n)_{ij}}{n!} + O(t^{n+1}),\quad\text{as } t\to0.
\end{equation}
\end{proof}
%%\begin{remark}
%%For any real $N\times N$-matrix $M$ we define $\|M\|_B:=\sup\limits_{|w|_B=1}\<Mw,w\>$.
%%It is easy to see that for every symmetric and positive definite $N\times N$-matrix $M$ and $L$ we have
%%\begin{align}
%%\|M\|_B>0,    \qquad
%%&\|M+L\|\leq\|M\|+\|L\|,\\
%%\|M\|_B\leq\|M\|,\qquad
%%&|M w|_B\leq \|M\|_B|w|_B,
%%\end{align}
%%where $\|\cdot\|$ is the spectral norm for square matrices.
%%We also note that for every $w\in\R^N$, if $h\in\{\ddd_{n-1}+1,\dots,\ddd_{n}\}$ then $|w_h|\leq|w|_B^{2n+1}$.
%%\end{remark}
%The following estimates are one of the main ingredients necessary to prove Proposition
%\ref{prop:gaussian_estimates} in the our settings \red{(?)}. \pur{intendevo nel caso backward, ma possiamo togliere}
\begin{lemma}\label{lemma:stimeH}
There exists a positive constant $C$ that only depends on $\m$, $B$ and $T_0$
such that, for every $i,j=1,\dots,d$ and $k=d +1,\dots,d + d_1$, %d+d_1$
%such that, for every $i,j=1,\dots,d$ and $k=d+1,\dots,\ddd_1$, %d+d_1$
%\red{(si capisce $H^{(s,v)}(t,T)^{-1}$?)}
\begin{align}
\big|\big( H^{(s,v)}(t,T)^{-1} x\big)_i\big|& \leq\frac{C}{\sqrt{T-t}} \big| D( \sqrt{T-t})^{-1}
e^{(T-t)B}x\big|, \label{stimeH1}\\ \big|(H^{(s,v)}(t,T)^{-1}\big)_{ij}\big|&
\leq\frac{C}{T-t}, \label{stimeH2}\\ \big|\big( H^{(s,v)}(t,T)^{-1} x\big)_k\big|&
\leq\frac{C}{(T-t)^{\frac{3}{2}}} \big| D( \sqrt{T-t})^{-1}
e^{(T-t)B}x\big|,\label{stimeH3}\\ \big|\big(H^{(s,v)}(t,T)^{-1}\big)_{ik}\big|&
\leq\frac{C}{(T-t)^2}, \label{stimeH4}
\end{align}
for any $0<T<T_0$ and $(t,x)\in\SS_T$.
%$0<T\leq\TT$ and $(t,x)\in[0,T]\times\R^N$.
\end{lemma}
\begin{proof}
We prove the first inequality. Setting $\t= T-t$, we have
%$v= e^{-(T-t)B}x$ and
\begin{align*}
\left|\( H^{(s,v)}(t,T)^{-1} x\)_i\right| & = \frac{1}{\sqrt{\t }} \left|\Big( D( \sqrt{\t })
e^{\t B^*}   \C^{(s,v)}(t,T)^{-1} e^{\t B}x \Big)_i\right|\\
&\leq \frac{1}{\sqrt{\t }} \sum_{n=1}^N \Big| \big(D(\sqrt{\t }) e^{\t B^*} D( \sqrt{\t })^{-1}\big)_{in}\Big| \left\|D( \sqrt{\t })  \C^{(s,v)}(t,T)^{-1}
D( \sqrt{\t })\right\| \left| D( \sqrt{\t })^{-1} e^{\t B}x\right|.
%\leq& \frac{1}{\sqrt{\t }} \sum_{h=1}^N \frac{\sqrt{\t}}{\t^{\frac{2k+1}{2}}} \(e^{\t B^*}\)_{ih}
%\left\|D( \sqrt{\t })  \C^{(s,v)}(t,T)^{-1} D( \sqrt{\t })\right\|
%\left| D( \sqrt{\t })^{-1} e^{\t B}x\right|
\end{align*}
%where $\|\cdot\|$ is the spectral norm for matrices. %operator norm for matrices induced by the Euclidean norm on $\R^N$.
%We consider the middle term.
By Lemma \ref{asint1} there exists a positive constant $\d$
such that, if $0<\t<\d$, we have
\begin{align*}
\left\|D( \sqrt{\t })  \C^{(s,v)}(t,T)^{-1} D( \sqrt{\t })\right\| &\leq\sup\limits_{|
y|=1}\<D( \sqrt{\t })  \C^{(s,v)}(t,T)^{-1} D( \sqrt{\t }) y, y\>
%&=\sup\limits_{| y|=1}\< \C^{(s,v)}(t,T)^{-1} D( \sqrt{\t }) y,D( \sqrt{\t }) y\>
\\ &\leq 2\m \sup\limits_{| y|=1}\<
\C_0(\t)^{-1} D( \sqrt{\t }) y,D( \sqrt{\t }) y\>=2\m\|\C_0(1)^{-1} \|,
\end{align*}
where the last equality follows from \eqref{eq:covmatomog}.
%\red{since $\C_0(t)= D(\sqrt{t})\C_0(1)D(\sqrt{t})$ for every $t\geq0$ (see \cite[Proposition 2.3]{lanpol})}
If $\d\leq\t < \TT$, by equation \eqref{confronto0} we have
\begin{equation}
\left\|D( \sqrt{\t})  \C^{(s,v)}(t,T)^{-1} D( \sqrt{\t})\right\| \leq \m \left\|D(
\sqrt{\t}) \C(\t)^{-1} D( \sqrt{\t})\right\|,
\end{equation}
which is {bounded} by a constant that depends only on $\mu$, $\TT$ and $B$. %independent from $t,T,x,v$,
%being $[\d_1,\TT]$ compact.

In order to conclude the proof of \eqref{stimeH1}, %we have to estimate the term $|\(D( \sqrt{\t }) e^{\t B^*} D( \sqrt{\t})^{-1}\)_{in}|$, for $n=1,\dots,N$. L
we let $h_n $ be the only $h\in\{0,\dots,\rr\}$ such that
$\ddd_{h-1}+1\leq n\leq \ddd_h$. Then, by  Lemma \ref{terminiexp}, since $i\in\{1,\dots,d\}$, we obtain
\begin{equation}
\big(D( \sqrt{\t }) e^{\t B^*} D( \sqrt{\t })^{-1}\big)_{in} = D( \sqrt{\t })_{ii}\big(e^{\t
B^*}\big)_{in} D( \sqrt{\t })^{-1}_{nn} =  {{\t}^{\frac{1}{2}}} \(e^{\t B}\)_{ni} {\t^{-\frac{2h_n +1}{2}}}
%\intertext{(by Lemma \ref{terminiexp}, since $i\in\{1,\dots,d\}$ and $j\in\{\bar{d}_{h_n -1}+1,\dots,\bar{d}_{h_n }\}$)}
 %= O(\t^{\frac{1}{2}+h_n -\frac{2h_n +1}{2}})
=O(1) \qquad\text{as } \tau \to0.
\end{equation}
%by Lemma \ref{terminiexp}, since $i\in\{1,\dots,d\}$ and $n\in\{\bar{d}_{h_n -1}+1,\dots,\bar{d}_{h_n }\}$.
%This is sufficient to prove that $\(D( \sqrt{\t }) e^{\t B^*} D( \sqrt{\t})^{-1}\)_{in}$ is bounded by a constant for $\t\in[0,T]$
%and this concludes the proof of \eqref{stimeH1}.
%Therefore there exists a positive constant $C$ such that
%\begin{equation}
%\left|\( H^{(s,v)}(t,T)^{-1} x\)_i\right| \leq\frac{1}{\sqrt{\t }} C \left| D( \sqrt{\t })^{-1}
%e^{\t B}x\right|.
%\end{equation}

Estimate \eqref{stimeH2} follows from \eqref{stimeH1} choosing $x=\mathbf{e}_j$. %the $j$-th vector of the canonical base of $\R^N$.
Estimates \eqref{stimeH3} and \eqref{stimeH4} can be proved following the
same arguments, noticing that for $k=d+1,\dots,d+d_1$ we have $D(\t)_{kk}=\t^{3}$.
\end{proof}
%%%%%%%%%%%%%%%%%%%%%%%%
%                                                       %
%           APPENDIX Gaussian estimates         %
%                                                    %
%%%%%%%%%%%%%%%%%%%%%%%%
%\section{Gaussian estimates}\label{app:gaussian}
Finally, we provide Gaussian estimates for $\G^{(s,v)}(t,x;T,y)$ and its derivatives up to the
fourth order that will be used to study the H\"older regularity of the second order derivatives of
the fundamental solution
via the representation \eqref{FS}-\eqref{termini}. % we need
%We provide Gaussian estimates for $\G^{(s,v)}(t,x;T,y)$ and its derivatives up to the fourth order
%that will be used to study the $C_{B}^{2,\b}$-regularity of the fundamental solution. % we need
%%estimates on the third and
%% derivative of $\G^{(s,v)}(t,x;T,y)$\red{, unlike what already shown in similar
%%settings in \cite{difpas} (toglierei questa parte in rosso)}.
%%Indeed, in order to prove a
%%$C_{B}^{2,\b}$-regularity of the fundamental solution, we have to study \red{the first and second
%%order derivatives of $p$)}:
%%heuristically, by equations \eqref{FS} and \eqref{termini}, %one sees that
%%this involves third and fourth order derivative of the paramentrix.
%In this section we suppose that Assumptions \ref{coer}, \ref{ass:hypo} and \ref{ass:regularity}
%are satisfied and we fix $(s,v)\in\mathcal{S}_{\TT}$.
%We first give a preliminary result. %that is also used through the paper. %one of the key ingredient for the parametrix construction.
The following result can be proved as \cite[Proposition 3.5]{difpas}.
\begin{lemma}\label{polipar1}
For every $\b\ge0$ and $\e>0$ there exists a positive constant $C$,
only dependent on $\TT,\m,B,\e$ and $\beta$, %the vector $(\b_1,\dots,\b_N)$,
such that
%, $\x:= D(\sqrt{T-t})^{-1}\(y-e^{(T-t)B} x\)$, we have
\begin{equation}\label{polipar}
 |w_i|^{\b} \, \G^{(s,v)}(t,x;T,y) %\frac{\left|y-e^{(T-t)B} x\right|_B^\b}{|T-t|^{\frac{\b}{2}}}
\leq C\G^{\m+\e}(t,x;T, y), \qquad (T,y)\in \mathcal{S}_{\TT},\ (t,x)\in\SS_T,\ i=1,\dots,N,
\end{equation}
%for every $(t,x),(T, y)\in \mathcal{S}_{\TT}$, $t<T$,
where
\begin{equation}
 w=D(\sqrt{T-t})^{-1}\(y-e^{(T-t)B} x\).
\end{equation}
\end{lemma}
\begin{notation}
Let $\n=(\n_1,\dots,\n_N)\in\mathbb{N}_0^N$ be a multi-index. We define the $B$-length of $\n$ as
\begin{equation}
[\n]_B:=%\sum_{j=1}^\rr(2i+1)|\n^{[i]}|, \qquad
 \sum_{j=0}^{\rr}(2j+1)\sum_{i=\bar{d}_{j-1}+1}^{\bar{d}_j}\n_{i}. %\qquad {\bar{d}_j:= \sum_{k=0}^j d_k}.
\end{equation}
Moreover, as usual
%\begin{equation}
 $\p^\n_x=\p_{x_1}^{\n_1}\cdots \p_{x_N}^{\n_N}.$
%\end{equation}
\end{notation}
%We are now in position to state the following
Combining Lemmas \ref{lemma:stimeH} and \ref{polipar1} with \cite[Proposition 3.1, 3.6 and Lemma
5.2]{difpas}, some lengthy but straightforward computations show the following
\begin{proposition}\label{prop:gaussian_estimates}
We have%it holds that
\begin{equation}\label{eq:upper_gaussian}
\frac{1}{\m^N}\G^{\frac{1}{\m}}(t,x;T,y) \leq \G^{(s,v)}(t,x;T,y) \leq\m^N \G^\m(t,x;T,y).
\end{equation}
{for any $(T,y)\in \mathcal{S}_{\TT}$ and $(t,x)\in\SS_T$}. Moreover, for every $\e>0$ and $\n\in
\N^N_0$ with $[\n]_B\le 4$, there exists a positive constant $C$, only dependent on $\TT, \m, B$
and $\e$,
such that %for every $\n\in \N^N_0, [\n]_B\leq4$, it holds
\begin{align}
|\p_{x}^\n
\G^{(s,v)}(t,x;T,y)|&\leq\frac{C}{(T-t)^{\frac{[\n]_B}{2}}}\G^{\m+\e}(t,x;T,y),\label{derparam}\\
\left|\p_{x}^\n\G^{(s,v)}(t,x;T,y)-\p_{x}^\n\G^{(s,w)}(t,x;T,y)\right| &\leq C
\frac{|v-w|_B^\a}{(T-t)^{\frac{[\n]_B}{2}}} \G^{\m+\e}(t,x;T,y), \label{incrementiparam}
\end{align}
{for any $(T,y)\in \mathcal{S}_{\TT}$, $(t,x)\in\SS_T$} and $w\in\R^N$. %\red{Here $\alpha$ is the
%H\"older exponent of the coefficients of $\Ac$ in Assumption \ref{ass:regularity}.}
\end{proposition}
%\red{The proof of these estimates is based on the same techniques used in \cite[Proposition 3.1,
%3.6 and Lemma
%5.2]{difpas} for the case $[\n]_B=1,2$, using the results in Lemma \ref{lemma:stimeH} to deal with the different definition of $\G^{(s,v)}$. The extension to higher order derivatives is lengthy but straightforward.} % : the same steps can be followed in order to obtain the estimates \red{on the higher order derivatives}. %of the third and the fourth order derivatives.
%%%%%%%%%%%%%%%%%%%%%%%%
%                                                       %
%           APPENDIX Potential estimates        %
%                                                    %
%%%%%%%%%%%%%%%%%%%%%%%%
\section{Potential estimates}\label{appendix:potentialestimates}
{We study $\Phi=\Phi(t,x;T,y)$ in \eqref{eq:def_Phi} and its derivatives w.r.t. to the variables
$x_1,\dots,x_d$. To do so, we have to deal with some singular integrals. We follow the steps in
\cite[Section 5]{difpas}, but we remark that the estimates of Proposition \ref{lem:stime_pot}
extend the ones in the aforementioned paper to higher order derivatives. This is needed to prove
the optimal regularity of $\Phi(t,x;T,y)$ and thereafter of $\sol(t,x;T,y)$.}

%Throughout this section we suppose that Assumptions \ref{coer}, \ref{ass:hypo} and
%\ref{ass:regularity} are satisfied.
We set
\begin{equation}\label{eq:def_J}
J(t,x;\t;T,y):=\int_{\R^N} \param(t,x;\t,\y)\phi(\t,\y;T,y) d\y, \qquad
(T,y)\in\mathcal{S}_{\TT},\ (t,x)\in \mathcal{S}_T, \ \t\in]t,T[.
\end{equation}
\begin{proposition}\label{lem:stime_pot}
%We use the notation
%Let $\n\in\N^N_0$ be a multi-index such that $[\n]_B\leq4$. Then, f
For every $\e>0$, $\n\in\N^N_0$  with $[\n]_B\le 4$ and $0<\d<\a$, there exists a positive
constant $C$, only dependent on $N, \TT, \m, B, \d,\a$ and $\e$, such that,
\begin{align}
\left|J(t,x;\t;T,y)\right| &\leq \frac{C}{(T-\t)^{1-\frac{\a}{2}}}\G^{\m+\e}(t,x;T,y)
\label{eq:pot_est}\\
\left|\p_x^\n J(t,x;\t;T,y)\right| &\leq
\frac{C}{(T-\t)^{1-\frac{\d}{2}}(\t-t)^{\frac{[\n]_B-(\a-\d)}{2}}}\G^{\m+\e}(t,x;T,y), \label{derpot}
%%%\left|\p_{x_i}J(t,x;\t;T,y)\right|
%%%&\leq \frac{C}{(T-\t)^{1-\frac{\a}{2}}(\t-t)^{\frac{1}{2}}}\G^{\m+\e}(t,x;T,y)\\
%%%procedendo come nelle stime del potenziale delle derivate seconde mi sembra si possa mostrare anche:
%\left|\p_{x_i}J(t,x;\t;T,y)\right| &\leq
%\frac{C}{(T-\t)^{1-\frac{\d}{2}}(\t-t)^{\frac{1}{2}-\frac{\a-\d}{2}}}\G^{\m+\e}(t,x;T,y)
%\label{der1pot}\\ \left|\p_{x_i x_j}J(t,x;\t;T,y)\right| &\leq
%\frac{C}{(T-\t)^{1-\frac{\d}{2}}({\t-t})^{1-\frac{\a-\d}{2}}}\G^{\m+\e}(t,x;T,y),
%\label{der2pot}\\ \left|\p_{x_i x_j x_k}J(t,x;\t;T,y)\right| &\leq
%\frac{C}{(T-\t)^{1-\frac{\d}{2}}({\t-t})^{\frac{3}{2}-\frac{\a-\d}{2}}}\G^{\m+\e}(t,x;T,y),\label{der3pot}\\
%\left|\p_{x_i x_j x_k x_l}J(t,x;\t;T,y)\right| &\leq
%\frac{C}{(T-\t)^{1-\frac{\d}{2}}({\t-t})^{2-\frac{\a-\d}{2}}}\G^{\m+\e}(t,x;T,y),\\
%\left|\p_{x_h}J(t,x;\t;T,y)\right| &\leq
%\frac{C}{(T-\t)^{1-\frac{\d}{2}}({\t-t})^{\frac{3}{2}-\frac{\a-\d}{2}}}\G^{\m+\e}(t,x;T,y),\\
%\left|\p_{x_h x_i}J(t,x;\t;T,y)\right| &\leq
%\frac{C}{(T-\t)^{1-\frac{\d}{2}}({\t-t})^{2-\frac{\a-\d}{2}}}\G^{\m+\e}(t,x;T,y),
\end{align}
%for any $i,j,k,l=1,\dots,d$ and $k=\ddd+1,\dots,\ddd_1$.
for every $(T,y)\in \mathcal{S}_{\TT}$, $(t,x)\in \mathcal{S}_T$ and  $\tau\in]t,T[$.
\end{proposition}
\begin{proof}%[Proof of Proposition \ref{lem:stime_pot}]
The proof relies on Proposition \ref{proposition_serie}: \eqref{eq:pot_est} can be easily obtained
by applying estimate \eqref{eq:upper_gaussian} to $\param(t,x;\t,\y)$, estimate
\eqref{eq:stimephi} to $\phi(\t,\y;T,y)$ and the Chapman-Kolmogorov identity.

We provide a full proof of \eqref{derpot} in the case of $\p_x^\n=\p_{x_i x_j}$, with $i,j\leq d$,
the proof for higher order derivatives being analogous. The idea is to combine
\eqref{eq:stimephihold} with the techniques in \cite[Proposition 5.3]{difpas} and
\cite[Proposition 3.2]{pol}. Let $(t,x)\in \mathcal{S}_T$ and  $\tau\in]t,T[$ be fixed. %Also fix $i,j=1,\dots, d$.}
By estimates \eqref{derparam} and \eqref{eq:stimephi},
%By the bound estimates on $\p_{ij}\param,\phi$
%and the Lebesgue dominated convergence theorem,
we have
\begin{equation}
\p_{x_i x_j}J(t,x;\t;T,y)=\int_{\R^N} \p_{x_i x_j}\param(t,x;\t,\y)\phi(\t,\y;T,y) d\y.
\end{equation}
%for any $t<\t<T$. We fix \
We set $\tt=\frac{t+T}{2}$ %, the midpoint of $[t,T]$,
and consider two separate cases:\\
%%
%%[\tt,T]
%%
\vspace{2pt}\emph{Case $\tt<\t< T$.} By \eqref{derparam} and \eqref{eq:stimephi}, we have that %the bound estimates on $\p_{ij}\param$ and $\phi$, %
for every $\e>0$ and $0<\d<\a$ there exists a positive constant $C$ such that
\begin{align*}
|\p_{x_i x_j}J(t,x;\t;T,y)|
&\leq\int_{\R^N}\frac{C}{(T-\t)^{1-\frac{\a}{2}}(\t-t)}\G^{\m+\e}(t,x;\t,\y)\G^{\m+\e}(\t,\y;T,y)d\y
\intertext{(by the Chapman-Kolmogorov equation)}
&\leq\frac{C}{(T-\t)^{1-\frac{\a}{2}}(\t-t)}\G^{\m+\e}(t,x;T,y) \intertext{(since $T-\t<\t-t$)}
&\leq\frac{C}{(T-\t)^{1-\frac{\d}{2}}(\t-t)^{1-\frac{\a-\d}{2}}}\G^{\m+\e}(t,x;T,y).
\end{align*}
%%
%%[t,\tt]
%%
\vspace{2pt}\emph{Case $t<\t\le \tt$.}
%When $t<\t<\tt$
Here we need to handle with care the singularity of $\p_{x_i x_j}\param(t,x;\t,\y)$ for small
$\t-t$. {Note that in this case the following inequalities hold true:
\begin{equation}\label{eq:inequalities_chain}
 \tau-t \leq\frac{T-t}{2}\leq T-\tau < T-t.
\end{equation}}
%be more accurate since the the previous argument does not work.
We have
\begin{equation}
\p_{x_i x_j}J(t,x;\t;T,y)=K_1%(t,x;\t;T,y)
+K_2%(t,x;\t;T,y)
+K_3%(t,x;\t;T,y)
,
\end{equation}
where, setting $\xx=e^{(\t-t)B}x$,
\begin{align}
K_1%(t,x;\t;T,y)
&:=\int_{\R^N} \p_{x_i x_j}\G^{(\t,\y)}(t,x;\t,\y)\big(\phi(\t,\y;T,y)-\phi(\t,\xx;T,y)\big) d\y,\\
K_2%(t,x;\t;T,y)
&:= \phi(\t,\xx;T,y) \int_{\R^N} \(\p_{x_i x_j}\G^{(\t,\y)}(t,x;\t,\y)-\p_{x_i
x_j}\G^{(\t,v)}(t,x;\t,\y)\big|_{v=\xx}\)d\y ,\\
K_3%(t,x;\t;T,y)
&:= \phi(\t,\x;T,y) \int_{\R^N} \p_{x_i x_j}\G^{(\t,v)}(t,x;\t,\y)\big|_{v=\xx}d\y.
\end{align}
We first consider $K_1%(t,x;\t;T,y)
$. By \eqref{eq:stimephihold} %the local H\"older regularity of $\phi$ %
and \eqref{derparam}, % the upper bound of the second derivative of $\G^{(\t,\y)}$
for every $\e>0$ and $0<\d<\a$ there exists a positive constant $C$ such that
\begin{align}
|K_1%(t,x;\t;T,y)
|%\\
 &\leq \frac{C}{(T-\t)^{1-\frac{\d}{2}}} \int_{\R^N}\frac{|\y-\xx|^{\a-\d}_B}{(\t-t)}\G^{\m+\frac{\e}{2}}(t,x;\t,\y)
 \(\G^{\m+\e}(\t,\xx;T,y)+\G^{\m+\e}(\t,\y;T,y)\)d\y
\intertext{(by {\eqref{polipar}})} &\leq \frac{C}{(T-\t)^{1-\frac{\d}{2}}}
\int_{\R^N}\frac{1}{(\t-t)^{1-\frac{\a-\d}{2}}}\G^{\m+\e}(t,x;\t,\y)
\(\G^{\m+\e}(\t,\xx;T,y)+\G^{\m+\e}(\t,\y;T,y)\)d\y
 \intertext{{(integrating in $\y$ and by the Chapman-Kolmogorov identity)}} &\leq
\frac{C}{(T-\t)^{1-\frac{\d}{2}}(\t-t)^{1-\frac{\a-\d}{2}}}
\(\G^{\m+\e}(\t,\xx;T,y)+\G^{\m+\e}(t,x;T,y)\)
%\frac{C}{(T-\t)^{1-\frac{\d}{2}}(\t-t)^{\frac{\a-\d}{2}}}
%\bigg(\underbrace{\int_{\R^N}\G^{\m+\e}(t,x;\t,\y)d\y}_{=1} \G^{\m+\e}(\t,x;T,y)+\G^{\m+\e}(t,x;T,y)\bigg)
\intertext{{(by \eqref{eq:inequalities_chain})}} & \leq
\frac{C}{(T-\t)^{1-\frac{\d}{2}}(\t-t)^{1-\frac{\a-\d}{2}}} \G^{\m+\e}(t,x;T,y).
\end{align}
%As $\t\in\,]t,\tt\,]$, we have $\frac{T-t}{2}\leq T-\tau < T-t$. Therefore, we obtain
%Therefore, for any $\t\in\,]t,\tt\,]$,
%\begin{equation}
%|K_1%(t,x;\t;T,y)
%|\leq \frac{C}{(T-\t)^{1-\frac{\d}{2}}(\t-t)^{1-\frac{\a-\d}{2}}}
%\G^{\m+\e}(t,x;T,y).
%\end{equation}
Consider now $K_2$. By {\eqref{eq:stimephi}} and \eqref{incrementiparam}, %the estimates on the parametrix and the boundedness of $\phi$
we obtain
\begin{align*}
|K_2%(t,x;\t;T,y)
| &\leq C{\frac{\G^{\m+\e}(\t,\xx;T,y)}{(T-\t)^{1-\frac{\a}{2}}}}
\int_{\R^N}\frac{|\y-\xx|^{\a}_B}{\t-t}\G^{\m+\e}(t,x;\t,\y)d\y\\ \intertext{(by {\eqref{polipar}
and integrating in $\eta$})} &\leq
% C \int_{\R^N}\frac{1}{(T-\t)^{1-\frac{\a}{2}}(\t-t)^{1-\frac{\a}{2}}}\G^{\m+\e}(t,x;\t,\y)\G^{\m+\e}(\t,\y;T,y)d\y\\
\frac{C}{(T-\t)^{1-\frac{\a}{2}}(\t-t)^{1-\frac{\a}{2}}} \G^{\m+\e}(\t,\xx;T,y)
%\intertext{(by the Chapman-Kolmogorov equation and since $\t-t<T-\t$)}
\intertext{{(again by \eqref{eq:inequalities_chain})}}
%&\leq \frac{C}{(T-\t)^{1-\frac{\d}{2}}(\t-t)^{1-\frac{2 \red{(?!)}\a-\d}{2}}} \G^{\m+\e}(t,x;T,y).
&{\leq \frac{C}{(T-\t)^{1-\frac{\d}{2}}(\t-t)^{1-\frac{\a-\d}{2}}} }\G^{\m+\e}(t,x;T,y).
\end{align*}
%for any $\t\in\,]t,\tt\,]$.
Finally, $K_{3}=0$ since %, we observe that Lebesgue dominated convergence theorem yields %by \eqref{der2param} and the Lebesgue theorem
  $$\int_{\R^N} \p_{x_i x_j}\G^{(\t,v)}(t,x;\t,\y)d\y=\p_{x_i x_j}\int_{\R^N}\G^{(\t,v)}(t,x;\t,\y)d\y=0$$
for any $v\in\rn$.
%\begin{align*}
%K_3%(t,x;\t;T,y)
%&=\phi(\t,\xx;T,y) \p_{x_i x_j}\Big(\underbrace{\int_{\R^N} \G^{(\t,v)}(t,x;\t,\y) d\y}_{=1}\Big)\Big|_{v=\xx} =0 .
%\end{align*}
%In conclusion, we have proved that %for any $\e>0, 0<\d<\a$ there exists a positive constant $C$ such that for any $\t\in\,]t,T[$ we have
%\begin{equation}
%|\p_{x_i x_j}J(t,x;\t;T,y)|
%\le\frac{C}{(T-\t)^{1-\frac{\d}{2}}(\t-t)^{1-\frac{\a-\d}{2}}}\G^{\m+\e}(t,x;T,y),
%\end{equation}
%which is \eqref{derpot} in the case of $\p_x^\n=\p_{x_i x_j}$, with $i,j=1,\dots,d$.
\end{proof}

\bibliographystyle{siam}
%\bibliography{library}
\bibliography{BibTeX-Final}
%\end{footnotesize}

\end{document}